\documentclass[leqno,11pt]{article}
\usepackage[includehead,includefoot,margin=20mm]{geometry}
\usepackage{amsmath}
\usepackage{amsthm}
\usepackage{amsfonts}  
\usepackage{amssymb}
\usepackage{amsmath}
\usepackage{mathrsfs}
 \usepackage[french]{babel}   
\usepackage[T1]{fontenc}

\title{Cl\^oture int\'egrale et op\'erations de tores alg\'ebriques de complexit\'e un dans les vari\'et\'es affines}
\author{Kevin Langlois}
\date{}

\begin{document}
\maketitle

\theoremstyle{plain}
\newtheorem{theorem}{Th\'eor\`eme}[section]
\newtheorem{lemme}[theorem]{Lemme}
\newtheorem{proposition}[theorem]{Proposition}
\newtheorem{corollaire}[theorem]{Corollaire}
\newtheorem*{theorem*}{Théorème}

\theoremstyle{definition}
\newtheorem{definition}[theorem]{D\'efinition}
\newtheorem{rappel}[theorem]{}
\newtheorem{conjecture}[theorem]{Conjecture}
\newtheorem{exemple}[theorem]{Exemple}
\newtheorem{notation}[theorem]{Notation}

\theoremstyle{remark}
\newtheorem{remarque}[theorem]{Remarque}
\newtheorem{note}[theorem]{Note}

\small
\paragraph{Abstract}
In this paper, we are interested in multigraded affine algebras of complexity $1$.
One of the main results gives a description of integrally closed homogeneous ideals in terms 
of polyhedral divisors introduced by Altmann-Hausen. 
Another result allows us to compute effectively
the normalization of an affine variety with an algebraic torus
action of complexity one. We describe as well the integral closure of homogeneous ideals and
exhibit new examples of normal homogeneous ideals.
\paragraph{}
\normalsize
\section*{\centerline{Introduction}}
In this paper, we are interested in multigraded affine algebras of complexity $1$
over an algebraically closed field $\mathbf{k}$ of characteristic zero.

Using the convex geometry developped by Altmann-Hausen we obtain some new 
results on classical questions of commutative algebra. 
One of our main theorems gives  
a description of integrally closed homogeneous ideals in terms 
of polyhedral divisors, see Theorem $4.6$. 
Another result allows us to compute effectively  the normalization of an affine variety 
with an algebraic torus
action of complexity one. We describe as well the integral closure of homogeneous ideals,
see Theorem $2.4$, Theorem $4.2$ and exhibit new examples 
of normal homogeneous ideals, see Theorem $5.3$.

The following two classical examples illustrate well the issues that arise. Consider
the algebra of Laurent polynomials in $n$ variables
\begin{eqnarray*}
 L_{[n]} = L_{[n]}(\mathbf{k}):=\mathbf{k}\left[t_{1},t_{1}^{-1},t_{2},t_{2}^{-1},\ldots,t_{n},t_{n}^{-1}\right].
\end{eqnarray*}
Notice that $L_{[n]}$ is the coordinate ring of the affine variety $(\mathbf{k}^{\star})^{n}$.
Let $A$ be a subalgebra generated by a finite number of monomials, 
and let $E\subset \mathbb{Z}^{n}$ be the subset of exponents corresponding to the monomials
of $A$. Without loss of generality, we may suppose that $E$ generates the lattice $\mathbb{Z}^{n}$.
It is known [Ho] that the normalization of $A$ is the set of linear combinations of
all monomials with their exponents belonging to the rational cone $\omega\subset \mathbb{Q}^{n}$
generated by $E$. We have
\begin{eqnarray*}
\bar{A} = \bigoplus_{(m_{1},\,\ldots\,,\,m_{n})\in\,\omega\cap\mathbb{Z}^{n}}\mathbf{k}\, 
t_{1}^{m_{1}}\ldots\, t_{n}^{m_{n}}
\end{eqnarray*}
where $\bar{A}$ is the normalization of $A$. For instance, if $n = 1$ and if $A$
is the subalgebra generated by the monomials $t_{1}^{2}$ and $t_{1}^{3}$
then the normalization of $A$ is the polynomial ring $\mathbf{k}[t_{1}]$.

A similar problem arises for monomial ideals.
Assume that the algebra $A$ is normal. Let $I$ be an ideal of
$A$ generated by monomials. The convex hull in $\mathbb{Q}^{n}$
of all exponents appearing in $I$ is a polyhedron $P$ contained in 
$\omega$. This polyhedron $P$ satisfies $P+ \omega \subset P$.
The integral closure of $I$ is equal to
\begin{eqnarray*}
\bar{I} = \bigoplus_{(m_{1},\,\ldots\,,\,m_{n})\in\,P\cap\mathbb{Z}^{n}}\mathbf{k}\, t_{1}^{m_{1}}\ldots\, t_{n}^{m_{n}}.  
\end{eqnarray*}
We can determine $P$ by means of a finite system of monomials generating the ideal $I$. 
For instance, if $n = 2$, $A = \mathbb{C}[t_{1},t_{2}]$ and if $I$ is the ideal generated by the monomials 
$t_{1}^{3}$ and $t_{2}^{3}$ then 
\begin{eqnarray*}
\bar{I} = \left( t^{3}_{1},\,t_{1}^{2}t_{2},\,t_{1}t_{2}^{2},\,t_{2}^{3}\right ). 
\end{eqnarray*}
See [LeTe], [HS], [Va] for more details concerning
the integral closure of ideals; we recall below the definition. Note that the study of integrally closed
ideals allows us to find the normalization of a  blowing-up (see [KKMS]
for toric varieties and [Br] for spherical varieties). This is used as well in order
to describe ($\mathbb{T}$-equivariant) affine modifications
(see [KZ], [Du]).

By analogy with the monomial case presented above, we address more generally the actions of
algebraic tori of complexity $1$. Before
formulating our results, let us recall some notation.

An algebraic torus $\mathbb{T}$ of dimension $n$ is an algebraic group isomorphic 
to $(\mathbf{k}^{\star})^{n}$. Let $M$ be the character lattice of $\mathbb{T}$, and let
$A$ be an affine algebra over $\mathbf{k}$.  
Defining an algebraic action of $\mathbb{T}$ on $X = \rm Spec\,\it A$
 is the same as defining an $M$-grading of $A$. The complexity of the affine $M$-graded algebra
$A$ is the codimension of a generic $\mathbb{T}$-orbit in $X$. Note that the classical toric case
corresponds to the complexity $0$ case.

Let $A$ be a domain and let $I$ be an ideal of $A$.
An element $a\in A$ is said to be integral over $I$ if there exist
$r\in\mathbb{Z}_{>0}$ and $c_{i}\in I^{i}$, $i = 1,\ldots, r$ such
that $a^{r} + \sum_{i = 1}^{r}c_{i}a^{r-i} = 0$. The integral closure $\bar{I}$ of 
the ideal $I$ is the set of all elements of $A$ that are integral over $I$.
It is known that $\bar{I}$ is an ideal [HS, $1.3.1$]. An ideal $I$ is integrally closed if
$I = \bar{I}$. Furthermore, $I$ is said to be
normal if for any positive integer $e$, the ideal $I^{e}$
is integrally closed. If $A$ is normal then this latter condition is equivalent 
to the normality of the Rees algebra $A[It] = A\oplus\bigoplus_{i\geq 1}I^{i}t^{i}$.
See [Ri] for more details.

The purpose of this paper is to answer the following questions: given an affine 
$M$-graded algebra $A$ of complexity $1$ over $\mathbf{k}$ and  homogeneous elements
$a_{1},\ldots,a_{r}\in A$ such that 
\begin{eqnarray*}
A = \mathbf{k}\left[a_{1},\ldots,\,a_{r}\right],
\end{eqnarray*} 
can one describe explicitly the normalization of $A$ in terms of the generators $a_{1},\ldots,a_{r}$? 
Furthermore if $A$ is normal and if $I$ is a homogenous ideal in $A$, can one describe effectively
 the integral closure of $I$ in terms of a given finite system of homogeneous generators of $I$? There
exists a connection between these two questions. Indeed the answer to the second can be
deduced from that to the first by examining the normalization of the Rees algebra $A[It]$
corresponding to $I$. 

To answer the first question, it is useful to attach an appropriate combinatorial object to a given normal 
$M$-graded algebra $A$. For instance, in the monomial case if $A$ is normal then
 the rational cone $\omega$ allows us to reconstruct $A$. 

Recall that a $\mathbb{T}$-variety is a normal variety endowed with an effective algebraic
$\mathbb{T}$-action. There exists
several combinatorial descriptions of affine $\mathbb{T}$-varieties. See [De], [AH] for 
arbitrary complexity, [KKMS], [Ti] for complexity $1$ and [FZ] for the case of surfaces.
Note that the description of [AH] is generalized in [AHS] for non-affine 
$\mathbb{T}$-varieties.
See also the survey article [AOPSV] 
for applications of the theory of $\mathbb{T}$-varieties. 

In this paper, we use the point of view of [AH] and [Ti]. To simplify things
we assume that $M = \mathbb{Z}^{n}$ and $\mathbb{T} = (\mathbf{k}^{\star})^{n}$. 
An affine  
$\mathbb{T}$-variety $X = \rm Spec\,\it A$ of complexity $1$ can be
described by its weight cone $\omega\subset\mathbb{Q}^{n}$ and by a polyhedral
divisor $\mathfrak{D}$ on a smooth algebraic curve $C$, whose coefficients
are polyhedra in $\mathbb{Q}^{n}$. For any element 
$m = (m_{1},\ldots,m_{n})$ of $\omega\cap\mathbb{Z}^{n}$, we have an evaluation
$\mathfrak{D}(m)$ belonging to the  $\mathbb{Q}$-linear space of
rational Weil divisors on $C$. Given a combinatorial data $(\omega, C,\mathfrak{D})$ 
one can construct an $M$-graded algebra
\begin{eqnarray*}
A[C,\mathfrak{D}] := \bigoplus_{m\in\omega\cap M}H^{0}(C,\mathcal{O}_{C}(\lfloor\mathfrak{D}(m)\rfloor ))\chi^{m},
\end{eqnarray*}
where $\chi^{m}$ is the Laurent monomial $t_{1}^{m_{1}}\ldots t_{n}^{m_{n}}$. See [AH] for definitions
and specific statements. One of the main results of this paper can be stated as follows (see
Theorem $2.4$).

\paragraph{Theorem.}{\em Let $C$ be a smooth algebraic curve. Consider a subalgebra
\begin{eqnarray*}
B = \mathbf{k}[C][f_{1}\chi^{s_{1}},\ldots,f_{r}\chi^{s_{r}}]\subset L_{[n]}(\mathbf{k}(C))
:=\mathbf{k}(C)\left[t_{1},t_{1}^{-1},\ldots,t_{n},t_{n}^{-1}\right] 
\end{eqnarray*}
such that
\begin{eqnarray*}
\rm Frac\, \it B = \rm Frac\,\it L_{[n]}(\mathbf{k}(C)),
\end{eqnarray*}
where $s_{i}\in\mathbb{Z}^{n}$, $\chi^{s_{i}}$ is the corresponding Laurent monomial, and
$f_{i}\in\mathbf{k}(C)^{\star}$. Then the  normalization of
$B$ is the algebra $A[C,\mathfrak{D}]$ with its weight cone $\omega$ generated by $s_{1},\ldots, s_{r}$
and with the following coefficient of the polyhedral divisor $\mathfrak{D}$ at the point $z\in C$ $\rm :$
\begin{eqnarray*}
\Delta_{z} = \left\{\,v\in\mathbb{Q}^{n},\,\left\langle s_{i}\,,\,v\right\rangle\geq -\nu_{z}(f_{i}),\,
i = 1,\ldots, r\,\right\},
\end{eqnarray*}
where $\nu_{z}(f_{i})$ is the order of $f_{i}$ at $z$.}
\paragraph{}
This theorem answers the first question.
It generalizes well known results for the case of affine 
surfaces with $\mathbf{k}^{\star}$-actions [FZ, 3.9, 4.6].
Note also that it may be applied 
\begin{enumerate}
 \item[(*)] to find a proper polyhedral divisor representing a (normal) subalgebra given 
by generators ;
\item[(**)] to find generators of an algebra represented by a proper polyhedral divisor:
the idea is to guess some generating set and apply (*).  
\end{enumerate} 

The answer to the second question is given by Theorem $4.2$.
It is known that the set of integrally closed homogeneous ideals of the affine toric
variety with weight cone $\omega$ is in bijective correspondence with the 
set of integral $\omega$-polyhedra contained in $\omega$ (see [KKMS] and section $3$).
We provide a similar correspondence for integrally closed homogeneous
ideals on affine $\mathbb{T}$-varieties of complexity $1$ (see Theorem $4.6$)
that is totally combinatorial when $C$ is affine (see Corollary $4.7$).

Any integrally closed homogeneous ideal $I$ of an affine 
$\mathbb{T}$-variety $X = \rm Spec\,\it A$ of complexity $1$ can be described by means of a pair 
$(P,\widetilde{\mathfrak{D}})$ where $P$ is an integral polyhedron in $\mathbb{Q}^{n}$. 
This polyhedron plays the same role as the Newton
polyhedron does in the monomial case. The polyhedral divisor $\widetilde{\mathfrak{D}}$
corresponds to the normalization of the Rees algebra of $I$. We give a geometric 
interpretation of
the coefficients of $\widetilde{\mathfrak{D}}$. Assume for instance that the weight cone
of $A$ is strongly convex, and let $\widetilde{\Delta}_{z}$ be the coefficient
of $\widetilde{\mathfrak{D}}$ at a point $z\in C$. Then Theorem $4.6$
provides conditions on the equations of facets of $\widetilde{\Delta}_{z}$ so that
$\widetilde{\mathfrak{D}}$ corresponds to the normalization of the Rees algebra $A[It]$.

We provide as well sufficient conditions 
on $(P,\widetilde{\mathfrak{D}})$ in order that $I$ be normal 
(see Theorem $5.3$). For
the case of non-elliptic affine $\mathbf{k}^{\star}$-surfaces,
we obtain a combinatorial proof for the normality of any integrally closed
invariant ideals of such surfaces.
As another application, we obtain the following new criterion of normality
which generalizes Reid-Roberts-Vitulli's Theorem [RRV, $3.1$] in
the case of complexity $0$ (see $5.5$).

\paragraph{Theorem.}{\em 
Let $n\geq 1$ be an integer and let $\mathbf{k}^{[n+1]} = \mathbf{k}[x_{0},\ldots, x_{n}]$
be the algebra of polynomials in $n+1$ variables over $\mathbf{k}$.
We endow $\mathbf{k}^{[n+1]}$ with the $\mathbb{Z}^{n}$-grading given by
\begin{eqnarray*}
\mathbf{k}^{[n+1]} = \bigoplus_{m = (m_{1},\ldots, m_{n})\in\mathbb{N}^{n}}A_{m},
\,\,\,\it where\,\,\,\it 
A_{m} = \mathbf{k}[x_{\rm \,0 \it }]\,x_{\rm 1\it }^{m_{\rm 1\it}}
\ldots x_{n}^{m_{n}}.
\end{eqnarray*}
For a homogeneous ideal $I$ of $\mathbf{k}^{[n+1]}$ the following are equivalent.
\begin{enumerate}
\item[\rm (i)] The ideal $I$ is normal;
\item[\rm (ii)] For any $e\in\{1,\ldots, n\}$, the ideal $I^{e}$ is 
integrally closed.
\end{enumerate}
} 

\paragraph{}
Let us give a brief summary of the contents of each section. 
In the first section, we recall some notions on tori actions
of complexity $1$ and on polyhedral divisors of Altmann-Hausen.
In the second section, we treat the normalization problem
for multigraded algebras and show Theorem $2.4$. The third section 
focuses on integrally closed monomial ideals.
In section $4$, we study the description 
by the pair $(P,\widetilde{\mathfrak{D}})$ for integrally
closed homogeneous ideals of affine $\mathbb{T}$-varieties. Finally, in the last section,   
we discuss the problem of normality in a special class.
\paragraph{}
Throughout this paper $\mathbf{k}$ is an algebraically closed field of characteristic zero.
By a variety we mean an integral separated scheme of finite type over $\mathbf{k}$.
\paragraph{}
{ \em Acknowledgments}:
The author is grateful to Mikhail Zaidenberg for his permanent encouragement and 
for posing the problem of section $2$.
We thank Ronan Terpereau and Mateusz Michalek for useful discussions.
We thank also the referees for pertinent remarks that allowed to improve
the presentation.

\section{$\mathbb{T}$-vari\'et\'es affines de complexit\'e un et g\'eom\'etrie convexe}
Nous rappelons ici les notions basiques sur les op\'erations de tores alg\'ebriques dont nous
aurons besoin par la suite. 
\begin{rappel}
Soit $N$ un r\'eseau de rang $n$ et $M = \rm Hom(\it N,\mathbb{Z})$
son r\'eseau dual. On note $N_{\mathbb{Q}} = \mathbb{Q}\otimes_{\mathbb{Z}}N$ et 
$M_{\mathbb{Q}} = \mathbb{Q}\otimes_{\mathbb{Z}}M$
les $\mathbb{Q}$-espaces vectoriels duaux associ\'es. Au r\'eseau $M$, on associe un tore 
alg\'ebrique $\mathbb{T}$ de dimension $n$ dont son alg\`ebre
des fonctions r\'eguli\`eres est d\'efinie par 
\begin{eqnarray*}
\mathbf{k}[\mathbb{T}] = \bigoplus_{m\in M}\mathbf{k}\,\chi^{m}.
\end{eqnarray*}
La famille $(\chi^{m})_{m\in M}$ satisfaisant les relations $\chi^{m}\cdot\chi^{m'} = \chi^{m+m'}$, 
pour tous $m,m'\in M$. Le 
choix d'une base de $M$ donne un isomorphisme entre $\mathbf{k}[\mathbb{T}]$ et l'alg\`ebre 
des polyn\^omes de Laurent \`a $n$ variables.
Chaque fonction $\chi^{m}$ s'interpr\`ete comme un caract\`ere de $\mathbb{T}$.
\end{rappel} 
\begin{rappel}
Soit $X$ une vari\'et\'e affine et supposons que $\mathbb{T}$ op\`ere alg\'ebriquement dans $X$. 
Alors cela induit une op\'eration de $\mathbb{T}$ dans 
$A := \mathbf{k}[X]$ d\'efinie par $(t\cdot f)(x) = f(t\cdot x)$ avec $t\in \mathbb{T}$, 
$f\in \mathbf{k}[X]$ et $x\in X$, faisant du $\mathbf{k}$-espace vectoriel $A$ un 
$\mathbb{T}$-module rationnel. Le $\mathbb{T}$-module $A$ admet une d\'ecomposition 
$A = \bigoplus_{m\in M}A_{m}$ en somme directe de sous-espaces vectoriels o\`u
pour tout $m\in M$, 
\begin{eqnarray*}
A_{m} = \left\{f\in A\,|\,\forall t\in \mathbb{T},\,t\cdot f = \chi^{m}(t)f\right\}\,.
\end{eqnarray*}
L'alg\`ebre $A$ est ainsi munie d'une $M$-graduation. R\'eciproquement, 
toute $M$-graduation de la $\mathbf{k}$-alg\`ebre $A$ est obtenue
par une op\'eration alg\'ebrique de $\mathbb{T}$ dans $X = \rm Spec\,\it A$. La partie
\begin{eqnarray*}
S := \left\{m\in M\,|\, A_{m}\neq \{0\}\right\}
\end{eqnarray*}
de $M$ contenant $0$ et stable par l'addition est appel\'ee \em semi-groupe des poids \rm de $A$. 
Puisque $A$ est de type fini sur $\mathbf{k}$, 
l'ensemble $S$ engendre un c\^one poly\'edral $\omega\subset M_{\mathbb{Q}}$ 
dit \em c\^one des poids \rm de $A$. 

L'op\'eration de $\mathbb{T}$ dans $X$ est fid\`ele si et seulement si $S$ n'est pas contenu dans un sous-r\'eseau propre de $M$. 
Dans ce cas, le c\^one $\omega$ est de dimension $n$ et
il existe un unique c\^one poly\'edral saillant $\sigma\subset N_{\mathbb{Q}}$ tel que
\begin{eqnarray*}
\omega = \left\{m\in M_{\mathbb{Q}}\,|\,\forall v\in\sigma,\, m(v)\geq 0\right\}\,.
\end{eqnarray*}
On dit que $\omega$ est le \em c\^one dual \rm de $\sigma$. On le note $\sigma^{\vee}$. 
Une \em $\mathbb{T}$-vari\'et\'e \rm est une vari\'et\'e 
normale munie d'une op\'eration alg\'ebrique fid\`ele de $\mathbb{T}$.

\end{rappel}
\begin{rappel}
Nous appelons \em complexit\'e \rm de l'op\'eration de $\mathbb{T}$ dans $X = \rm Spec\,\it A$, le degr\'e de 
transcendance de l'extension de corps $\mathbf{k}(X)^{\mathbb{T}}/\mathbf{k}$. La complexit\'e s'interpr\`ete 
g\'eom\'etriquement comme la codimension d'une orbite
en position g\'en\'erale. Lorsque l'op\'eration est fid\`ele, elle est \'egale \`a $\dim\, X - \dim\,\mathbb{T}$. 
Elle fut 
introduite dans [LV] pour les op\'erations de groupes alg\'ebriques r\'eductifs dans les espaces homog\`enes. 

Les $\mathbb{T}$-vari\'et\'es affines de complexit\'e $0$ sont les \em vari\'et\'es toriques affines. \rm 
Elles admettent une description bien connue en terme de c\^ones poly\'edraux saillants.
Pour plus d'informations, voir [CLS], [Da], [Fu], ainsi que [Od].
\paragraph{}
Etant donn\'es un sous-semi-groupe $E\subset M$ et un corps $K_{0}$,  
nous notons
\begin{eqnarray*}
K_{0}[E] = \bigoplus_{m\in E}K_{0}\,\chi^{m}.
\end{eqnarray*}
l'alg\`ebre du semi-groupe $E$ sur le corps $K_{0}$. Pour tout c\^one poly\'edral 
$\tau\subset M_{\mathbb{Q}}$, nous \'ecrivons $\tau_{M}:=\tau\cap M$. 
\end{rappel}
\paragraph{}
Nous rappelons une description combinatoire des $\mathbb{T}$-vari\'et\'es affines de complexit\'e $1$ 
obtenue dans [AH], [Ti]. Voir [FZ] pour 
la pr\'esentation de Dolgachev-Pinkham-Demazure $(D.P.D.)$ des 
$\mathbb{C}^{\star}$-surfaces affines complexes. 
Notons que 
l'approche donn\'ee dans [Ti] provient d'une description 
ant\'erieure des op\'erations de groupes alg\'ebriques r\'eductifs 
de complexit\'e $1$ [Ti $2$]. Dans [Vo], on donne un lien entre
[AH] et la description classique de [KKMS].

\begin{rappel}
Soit $C$ une courbe alg\'ebrique lisse et $\sigma\subset N_{\mathbb{Q}}$ un c\^one poly\'edral saillant. Une partie
$\Delta\subset N_{\mathbb{Q}}$ est un \em $\sigma$-poly\`edre \rm si $\Delta$ est somme de Minkowski de $\sigma$ et 
d'un polytope $Q\subset N_{\mathbb{Q}}$. On note $\rm Pol_{\sigma}(\it N_{\mathbb{Q}})$ le semi-groupe 
ab\'elien des $\sigma$-poly\`edres de loi la somme de Minkowski et d'\'el\'ement neutre $\sigma$. 

Pour $\Delta\in\rm Pol_{\sigma}(\it N_{\mathbb{Q}})$, on d\'efinit la fonction $h_{\Delta}:\sigma^{\vee}\rightarrow \mathbb{Q}$
 par $h_{\Delta}(m)  = \min_{v\in\Delta}m(v)$, pour tout $m\in\sigma^{\vee}$. On dit que $h_{\Delta}$
est la \em fonction de support \rm de $\Delta$. Elle est identiquement nulle si et seulement si $\Delta = \sigma$. Pour tous $m,m'\in\sigma^{\vee}$, 
on a l'in\'egalit\'e de sous-additivit\'e 
\begin{eqnarray}
h_{\Delta}(m)+h_{\Delta}(m')\leq h_{\Delta}(m+m').
\end{eqnarray}
Un \em diviseur $\sigma$-poly\'edral \rm sur $C$ est une somme formelle
\begin{eqnarray*}
\mathfrak{D} = \sum_{z\in C}\Delta_{z}\,.\,z
\end{eqnarray*} 
 o\`u chaque $\Delta_{z}$ appartient \`a $\rm Pol_{\sigma}(\it N_{\mathbb{Q}})$ avec pour presque 
\footnote{ Cela signifie qu'il existe un sous-ensemble fini $F\subset C$ tel que pour tout $z\in C-F$, 
$\Delta_{z} = \sigma$.} tout $z\in C$, $\Delta_{z} = \sigma$. On note $\rm supp(\it\mathfrak{D}\rm)$, dit
\em support \rm de $\mathfrak{D}$, l'ensemble des points $z\in C$ tels que $\Delta_{z}\neq \sigma$.
On dit que $\mathfrak{D}$ est \em propre \rm lorsque 
pour tout $m\in\sigma^{\vee}_{M}$, l'\'evaluation 
\begin{eqnarray*}
\mathfrak{D}(m) := \sum_{z\in C}h_{\Delta_{z}}(m)\,.\, z
\end{eqnarray*}
est un $\mathbb{Q}$-diviseur semi-ample et abondant (big) pour $m$ appartenant \`a l'int\'erieur relatif de 
$\sigma^{\vee}$.
\end{rappel}
\begin{rappel}
 La propret\'e de $\mathfrak{D}$ est d\'ecrite en distinguant les deux cas suivants [AH, $2.12$].
\begin{enumerate}
 \item[(i)]
Lorsque $C$ est une courbe projective lisse, le diviseur poly\'edral $\mathfrak{D}$ est propre si et seulement 
si pour tout vecteur $m\in\sigma^{\vee}_{M}$, on a $\rm deg(\it \mathfrak{D}(m)\rm ) \geq 0$ et si 
$\rm deg(\it \mathfrak{D}(m)\rm ) = 0$ alors $m$ appartient 
au bord de $\sigma^{\vee}$ et un multiple entier non nul de $\mathfrak{D}(m)$ 
est principal;
\item[(ii)]Lorsque $C$ est une courbe affine lisse, aucune condition n'est impos\'ee sur l'\'evaluation de $\mathfrak{D}$. 
\end{enumerate}
\end{rappel}
\begin{notation}
D'apr\`es l'in\'egalit\'e (1) dans $1.4$, si $\mathfrak{D}$ est un diviseur $\sigma$-poly\'edral sur $C$ alors pour
tous vecteurs $m,m'\in\sigma^{\vee}_{M}$, l'application 
\begin{eqnarray*}
H^{0}(C,\mathcal{O}_{C}(\lfloor\mathfrak{D}(m)\rfloor ))\times H^{0}(C,\mathcal{O}_{C}(\lfloor\mathfrak{D}(m')\rfloor ))\rightarrow
H^{0}(C,\mathcal{O}_{C}(\lfloor\mathfrak{D}(m+m')\rfloor )),\,\, (f,g)\mapsto f\cdot g  
\end{eqnarray*}
 est bien d\'efinie. Le $\mathbb{T}$-module 
\begin{eqnarray*}
A[C,\mathfrak{D}] := \bigoplus_{m\in\sigma^{\vee}_{M}}H^{0}(C,\mathcal{O}_{C}(\lfloor\mathfrak{D}(m)\rfloor ))\chi^{m},
\end{eqnarray*}
est donc une alg\`ebre $M$-gradu\'ee. On la notera $A[\mathfrak{D}]$ 
lorsque la mention de $C$ est claire. 
Pour tout \'el\'ement homog\`ene $f\chi^{m}$ de $\mathbf{k}(C)[M]$, 
nous sous-entendons que $f\in k(C)^{\star}$ et que $\chi^{m}$ 
est un caract\`ere du tore $\mathbb{T}$. 
\end{notation}
\paragraph{}
Le th\'eor\`eme suivant d\'ecrit les alg\`ebres affines 
normales $M$-gradu\'ees de complexit\'e $1$
et de c\^one des poids $\sigma^{\vee}$ (voir [AH], [Ti]).
\begin{theorem}
\begin{enumerate}
 \item[\rm (i)] 
Si $\mathfrak{D}$ est un diviseur $\sigma$-poly\'edral propre sur une courbe alg\'ebrique lisse $C$ 
alors l'alg\`ebre multigradu\'ee $A = A[C,\mathfrak{D}]$
est normale, de type fini sur $\mathbf{k}$ et sa $M$-graduation fait de $X = \rm Spec\, \it A\rm$ 
une $\mathbb{T}$-vari\'et\'e de complexit\'e $1$. R\'eciproquement, toute alg\`ebre des fonctions
r\'eguli\`eres d'une $\mathbb{T}$-vari\'et\'e affine de complexit\'e $1$ est isomorphe comme alg\`ebre 
$M$-gradu\'ee \`a $A[C,\mathfrak{D}]$, o\`u $C$ est une courbe alg\'ebrique lisse et
$\mathfrak{D}$ est un diviseur poly\'edral propre sur $C$.
\item[\rm (ii)]
Plus pr\'ecis\'ement, toute sous-alg\`ebre multigradu\'ee $A\subset\mathbf{k}(C)[M]$ 
normale, de type fini sur $\mathbf{k}$, satisfaisant $A_{0} = \mathbf{k}[C]$,
de m\^eme corps des fractions que $A$ et ayant $\sigma^{\vee}$ pour c\^one 
des poids, est \'egale \`a $A[C,\mathfrak{D}]$ pour un unique diviseur
$\sigma$-poly\'edral $\mathfrak{D}$ propre sur $C$. 
\end{enumerate}
\end{theorem}
\begin{notation}
Si $\mathfrak{D}$ est un diviseur poly\'edral
propre sur $C$ alors on note par $X[C,\mathfrak{D}]$ la $\mathbb{T}$-vari\'et\'e 
affine $\rm Spec \,\it A[C,\mathfrak{D}]\rm$ correspondante.
\end{notation}
\paragraph{}
L'assertion suivante est connue [AH, \S 8], [Li, Theorem 1.4(3)]. Elle permet de comparer deux donn\'ees combinatoires 
$(C, \sigma, \mathfrak{D})$ et $(C', \sigma', \mathfrak{D}')$ donnant des $\mathbf{k}$-alg\`ebres $M$-gradu\'ees isomorphes.
\begin{theorem}
Soient $C,C'$ des courbes alg\'ebriques lisses et $\mathfrak{D} , \mathfrak{D}'$ des diviseurs poly\'edraux propres
respectivement sur $C,C'$ selon des c\^ones poly\'edraux saillants $\sigma , \sigma'\subset N_{\mathbb{Q}}$. Alors 
$X[C,\mathfrak{D}]$ et $X[C',\mathfrak{D}']$ sont $\mathbb{T}$-isomorphes si et seulement si il existe un automorphisme  
de r\'eseau $F:N\rightarrow N$ tel que\footnote{Un automorphisme de r\'eseau $F:N\rightarrow N$ induit un 
automorphisme de l'espace vectoriel $N_{\mathbb{Q}}$ que l'on note $F_{\mathbb{Q}}$.} $F_{\mathbb{Q}}(\sigma) = \sigma'$,
 un isomorphisme $\phi : C\rightarrow C'$ de courbes alg\'ebriques,
$v_{1},\ldots,v_{r}\in N$ et $f_{1},\ldots,f_{r}\in k(C)^{\star}$ des fonctions rationnelles non nulles tels que pour tout $m\in\sigma'^{\,\vee}_{M}$,
\begin{eqnarray*}
\phi^{\star}(\mathfrak{D}')(m) = F_{\star}(\mathfrak{D})(m) + \sum_{i=1}^{r}m(v_{i})\cdot\rm div( \it f_{i}\rm )
\end{eqnarray*} 
avec 
\begin{eqnarray*}
\mathfrak{D} = \sum_{z\in C}\Delta_{z}\,.\, z,\,\,\mathfrak{D}' = \sum_{z'\in C'}\Delta'_{z'}\,.\, z',\,\,
\phi^{\star}(\mathfrak{D}') = \sum_{z'\in C'}\Delta'_{z'}\,.\,\phi^{-1}(z')\,\, \rm et \it\,\,
F_{\star}(\mathfrak{D}) = \sum_{z\in C}(F_{\mathbb{Q}}(\rm \Delta_{\it z}) + \sigma )\,.\,\it z.
\end{eqnarray*}
\end{theorem}  
\paragraph{}
La proposition suivante montre que la fonction de support 
d'un $\sigma$-poly\`edre est lin\'eaire par mor\c ceaux.
La preuve de ce r\'esultat est ais\'ee et laiss\'ee aux lecteurs.
\begin{lemme}
Soit $\sigma\subset N_{\mathbb{Q}}$ un c\^one poly\'edral saillant, $\Delta\in\rm Pol_{\sigma}(\it N_{\mathbb{Q}})$ 
un poly\`edre et $S\subset \Delta$
une partie non vide. Notons $V(\Delta)$ l'ensemble des sommets de $\Delta$.
Alors $\Delta = \rm Conv(\it S\rm ) \it +\sigma$ si et seulement si 
pour tout $m\in\sigma^{\vee}$, $h_{\Delta}(m) = \min_{v\in S}m(v)$.
En particulier, pour tout $m\in\sigma^{\vee}$, on a $h_{\Delta}(m) = \min_{v\in V(\Delta )}m(v)$.
\end{lemme}
\paragraph{}
La terminologie suivante est classique pour les $\mathbb{C}^{\star}$-surfaces affines complexes [FZ]. Elle est 
introduite dans [Li] pour les $\mathbb{T}$-vari\'et\'es affines de complexit\'e $1$. 
\begin{definition}
Soit $X$ une vari\'et\'e affine. Une op\'eration alg\'ebrique de $\mathbb{T}$ de complexit\'e $1$ dans $X$ est dite 
\em elliptique \rm si 
l'alg\`ebre multigradu\'ee correspondante 
\begin{eqnarray*}
A = \mathbf{k}[X] = \bigoplus_{m\in M}A_{m}\,\,\,\,\rm avec \,\,\it A_{m} = 
\left\{f\in A\,|\,\forall t\in \mathbb{T},\,t\cdot f = \chi^{m}(t)f\right\}
\end{eqnarray*}
v\'erifie $A_{0} = \mathbf{k}$. Dans ce cas, on dit que l'alg\`ebre $M$-gradu\'ee $A$
est elliptique.
\end{definition} 
\begin{remarque} 
Consid\'erons $A = A[C,\mathfrak{D}]$. Alors $A$ est elliptique si
et seulement si $C$ est projective. L'ellipticit\'e de $A$ donne des 
contraintes g\'eom\'etriques sur la vari\'et\'e $X = \rm Spec\,\it A$. 
En effet, si $C$ est affine alors $X$ est toro\"idale [KKMS] et n'admet donc
que des singularit\'es toriques. 
Tandis que lorsque $C$ est projective de genre $g\geq 1$, la vari\'et\'e $X$ 
a au moins une singularit\'e qui n'est pas rationnelle [LS, Propositions 5.1, 5.5].     
\end{remarque}
La proposition suivante peut \^etre vue comme une g\'en\'eralisation de [FZ, Lemma 1.3(a)].
\begin{proposition}
Soit $X$ une $\mathbb{T}$-vari\'et\'e affine de complexit\'e $1$ et notons $\sigma\subset N_{\mathbb{Q}}$ 
le dual du c\^one des poids de 
l'alg\`ebre multigradu\'ee
\begin{eqnarray*}
A = \mathbf{k}[X] = \bigoplus_{m\in M}A_{m}\,\,\,\rm avec \,\,\,\it A_{m} = 
\left\{f\in A\,|\,\forall t\in T,\,t\cdot f = \chi^{m}(t)f\right\}
\end{eqnarray*}
obtenue par l'op\'eration de $\mathbb{T}$. Les assertions suivantes sont vraies.
\begin{enumerate}
\item[\rm (i)] Si l'op\'eration de $\mathbb{T}$ n'est pas elliptique alors 
le semi-groupe des poids de $\mathbf{k}[X]$ est $\sigma^{\vee}_{M}$ 
et pour tout $m\in\sigma^{\vee}_{M}$, le $A_{0}$-module $A_{m}$ est localement libre de rang $1$;
\item[\rm (ii)] Si l'op\'eration est elliptique alors $\sigma\neq\{0\}$; 
\item[\rm (iii)] L'intersection du sous-ensemble
\begin{eqnarray*}
L = \{m\in\sigma^{\vee}_{M}\,|\, A_{m} = \{0\}\}\subset M_{\mathbb{Q}}
\end{eqnarray*} 
avec toute droite vectorielle rencontrant l'int\'erieur relatif de $\sigma^{\vee}$ est finie.
\end{enumerate}
\end{proposition}
\begin{proof}
Si l'op\'eration de $\mathbb{T}$ n'est pas elliptique alors par le th\'eor\`eme $1.7$, on peut 
supposer qu'il existe une courbe alg\'ebrique affine lisse $C$ et un diviseur $\sigma$-poly\'edral propre $\mathfrak{D}$
sur $C$ tels que pour tout $m\in\sigma^{\vee}_{M}$, 
\begin{eqnarray}
A_{m} = H^{0}(C,\mathcal{O}_{C}(\lfloor\mathfrak{D}(m)\rfloor ))\chi^{m}\,\,\,\,\rm  et \it\,\,\,\, \mathbf{k}[X] = \bigoplus_{m\in\sigma^{\vee}_{M}}A_{m}.
\end{eqnarray}
Comme $C$ est affine, pour tout $m\in \sigma^{\vee}_{M}$, on a $A_{m}\neq \{0\}$. D'o\`u l'assertion $\rm (i)$. 

Si l'op\'eration est elliptique alors on
peut supposer que $\mathbf{k}[X]$ v\'erifie $(2)$ avec $C$ une courbe projective lisse de genre $g$. Si $\sigma = \{0\}$ alors 
$\sigma^{\vee} = M_{\mathbb{Q}}$. Pour tout $m\in \sigma^{\vee}_{M}$, on a par propret\'e: 
$\rm deg(\it \mathfrak{D}(m)\rm ) > 0$. L'\'egalit\'e $\mathfrak{D}(0) = 0$ donne une contradiction.
D'o\`u l'assertion $\rm (ii)$. 

Soit $m\in \sigma^{\vee}_{M}$ appartenant \`a l'int\'erieur relatif de $\sigma^{\vee}$.
Alors il existe $r_{0}\in\mathbb{Z}_{>0}$ tel que pour tout $z\in\rm supp(\it\mathfrak{D}\rm)$ et 
tout sommet $v$ de $\Delta_{z}$, $r_{0}v\in N$. Ainsi, par le lemme $1.10$,
\begin{eqnarray}
\mathfrak{D}(r_{0}m) = \lfloor\mathfrak{D}(r_{0}m)\rm )\rfloor.
\end{eqnarray}
Par la propret\'e de $\mathfrak{D}$, on peut supposer que 
\begin{eqnarray*}
\rm deg(\it \mathfrak{D}(r_{\rm 0\it }m)\rm )\it  = \rm deg(\it \lfloor\mathfrak{D}(r_{\rm 0 \it}m)\rfloor\rm ) >\it  d +g-\rm 1
\end{eqnarray*}
o\`u $d$ est le cardinal de $\rm supp(\it\mathfrak{D}\rm)$. Donc pour tout $r\in\mathbb{N}$,
\begin{eqnarray}
\rm deg(\it \lfloor\mathfrak{D}((r_{\rm 0 \it}+r)m)\rfloor\rm )\it = \rm deg(\it \lfloor\mathfrak{D}(r_{\rm 0 \it}m)\rfloor\rm )\it + \rm deg(\it \lfloor\mathfrak{D}(rm)\rfloor\rm )\geq deg(\it \lfloor\mathfrak{D}(r_{\rm 0\it}m)\rfloor\rm )\it - d \rm >\it g-\rm 1.
\end{eqnarray}
Comme $\sigma\neq\{0\}$, la demi-droite $\mathbb{Q}_{\leq 0}\cdot m$ ne rencontre
$\sigma^{\vee}$ qu'en l'origine [CLS, Exercice 1.2.4].
Par le th\'eor\`eme de Riemann-Roch, $(4)$ donne l'inclusion 
\begin{eqnarray*}
 \mathbb{Q}\cdot m\,\cap\, L\subset \{0,m,2m,\ldots,(r_{0}-1)m\}\,.
\end{eqnarray*}
D'o\`u le r\'esultat.
\end{proof}
\begin{remarque}
Notons $\sigma = \mathbb{Q}_{\geq 0}^{2}$.
L'exemple du diviseur poly\'edral propre 
\begin{eqnarray*}
\mathfrak{D} = \Delta_{0}\cdot \{0\}+ \Delta_{1}\cdot\{1\} + \Delta_{\infty}\cdot\{\infty\},\,\,\,
\Delta_{0} = \left(\frac{1}{2},0\right) + \sigma,\,\,\, \Delta_{1} = \left(-\frac{1}{2},0\right) + \sigma,\,\,\,\Delta_{\infty} = [(1,0),(0,1)]+\sigma
\end{eqnarray*} 
sur $\mathbb{P}^{1}$
montre qu'en g\'en\'eral, il existe des demi-droites vectorielles contenues dans le bord de $\sigma^{\vee}$ qui 
rencontrent $L$ avec une infinit\'e de points. En effet, dans cet exemple, pour tout $r\in\mathbb{N}$, on a
\begin{eqnarray*}
H^{0}(\mathbb{P}^{1},\mathcal{O}_{\mathbb{P}^{1}}(\lfloor\mathfrak{D}(2r+1,0)\rfloor ))\chi^{(2r+1,0)} = \{0\}.
\end{eqnarray*}
\end{remarque}
\section{Normalisation des alg\`ebres affines multigradu\'ees de complexit\'e un}
Le but de cette section est de d\'ecrire explicitement la normalisation
des alg\`ebres affines multigradu\'ees de complexit\'e $1$.

Dans le cas de la complexit\'e $0$, toute alg\`ebre affine $M$-gradu\'ee
est r\'ealis\'ee comme sous-alg\`ebre $\mathbb{T}$-stable 
de $\mathbf{k}[M]$. La normalisation de $A$ est d\'etermin\'ee par son
c\^one des poids [Ho].
Par analogie, toute alg\`ebre affine $M$-gradu\'ee de complexit\'e $1$ est plong\'ee dans une 
alg\`ebre $\mathbf{k}(C)[M]$ o\`u $C$ est une courbe alg\'ebrique lisse. 
Nous rappelons ceci dans le paragraphe suivant. 
\begin{rappel}
Soit $A = \bigoplus_{m\in M}A_{m}$ 
une alg\`ebre int\`egre $M$-gradu\'ee de type fini sur $\mathbf{k}$.
Notons $K$ son corps des fractions. On suppose que $\rm tr.deg_{\it \mathbf{k}}\,\it K\rm^{\it \mathbb{T}} = 1$. 
Sans perte de g\'en\'eralit\'e, on peut supposer que la $M$-graduation de $A$ 
est fid\`ele. Alors pour tout vecteur $m\in M$,
\begin{eqnarray*}
K_{m} := \{f/g\,|\, f\in A_{m+e},\,g\in A_{e},\,g\neq 0\}\subset K
\end{eqnarray*}
est un sous-espace vectoriel de dimension $1$ sur $K_{0} = K^{\mathbb{T}}$. Le choix d'une base de $M$ nous permet 
de construire une famille $(\chi^{m})_{m\in M}$ d'\'el\'ements de $K^{\star}$ v\'erifiant pour tous 
$m,m'\in M$,
\begin{eqnarray*} 
K_{m} = K_{0}\,\chi^{m}\,\,\,\,\rm et \,\,\,\,\it\chi^{m}\cdot\chi^{m'} = \chi^{m+m'}.
\end{eqnarray*}
Quitte \`a remplacer $A_{0}$ par $\bar{A_{0}}$, on peut 
supposer que $A_{0}$ est normale. Notons que l'alg\`ebre $A_{0}$ est de type fini sur $\mathbf{k}$. 
Soit $\widetilde{C}$ la courbe compl\`ete lisse sur $\mathbf{k}$ obtenue \`a partir de l'ensemble 
des valuations discr\`etes de $K_{0}/\mathbf{k}$, de sorte que
$K_{0} = \mathbf{k}(\widetilde{C})$. 

Si l'op\'eration de $\mathbb{T}$ dans $X = \rm Spec\,\it A$ n'est pas elliptique alors $K_{0}$ est
le corps des fractions de $A_{0}$. En effet, si $b\in K_{0}$ alors $b$ est un \'el\'ement alg\'ebrique sur
$\rm Frac\,\it A_{\rm 0}$. Il existe donc $a\in A_{0}$ non nul tel que $ab$ soit entier sur $A_{0}$. Par
normalit\'e de $A_{0}$, on a $ab \in \bar{A}\cap K_{0} = A_{0}$ et donc $K_{0} \subset \rm Frac\,\it A_{\rm 0}$.
L'inclusion r\'eciproque est imm\'ediate.
Dans tous les cas, il existe un unique ouvert non vide 
$C\subset \widetilde{C}$ tel que 
\begin{eqnarray*}
A_{0} = \mathbf{k}[C]\,\,\,\,\rm et \,\,\,\,\it A\subset \bigoplus_{m\in M}K_{m} = 
\mathbf{k}(C)[M]\,.
\end{eqnarray*}
Par le th\'eor\`eme $1.7$, 
l'\'egalit\'e $K = \rm Frac\,\it \mathbf{k}(C)[M]$ 
implique que $\bar{A} = A[C,\mathfrak{D}]$, pour un unique
diviseur poly\'edral propre $\mathfrak{D} = \sum_{z\in C}\Delta_{z}\cdot z$. 
\paragraph{}
Fixons un syst\`eme de g\'en\'erateurs homog\`enes 
\begin{eqnarray*}
A = \mathbf{k}[C][f_{1}\chi^{m_{1}},\ldots,f_{r}\chi^{m_{r}}]
\end{eqnarray*}
avec $f_{1},\ldots,f_{r}$ des fonctions rationnelles non nulles sur $C$ et $m_{1},\ldots,m_{r}$ des \'el\'ements de $M$.
Il s'agit de d\'eterminer les poly\`edres $\Delta_{z}$ en fonction de  
$((f_{1},m_{1}),\ldots,(f_{r},m_{r}))$.  Dans [FZ], on donne la pr\'esentation $D.P.D.$ de $\bar{A}$ 
pour le cas des surfaces affines complexes avec op\'erations paraboliques et hyperboliques de $\mathbb{C}^{\star}$. 
Nous rappelons ces r\'esultats dans le corollaire $2.7$.
\end{rappel}
Le lemme \'el\'ementaire suivant sera utilis\'e dans la preuve du th\'eor\`eme $2.4$.
La d\'emonstration de ce r\'esultat est laiss\'ee aux lecteurs.
\begin{lemme}
Soit $\sigma\subset N_{\mathbb{Q}}$ un c\^one poly\'edral saillant et $\Delta_{1},\Delta_{2}\in\rm Pol_{\sigma}(\it N_{\mathbb{Q}})$
des $\sigma$-poly\`edres. 
Alors $\Delta_{1} = \Delta_{2}$ si et seulement si pour tout $m\in\sigma^{\vee}_{M}$ appartenant \`a 
l'int\'erieur relatif de $\sigma^{\vee}$, $h_{\Delta_{1}}(m) = h_{\Delta_{2}}(m)$.
\end{lemme}

\begin{notation}
Soient $C$ une courbe alg\'ebrique lisse et $f = (f_{1}\chi^{m_{1}},\ldots,f_{r}\chi^{m_{r}})$ 
un $r$-uplet d'\'el\'ements homog\`enes de $\mathbf{k}(C)[M]$ 
tels que $\sum_{i = 1}^{r}\mathbb{Q}\,m_{i} = M_{\mathbb{Q}}$. Posons 
$\sigma \subset N_{\mathbb{Q}}$ le c\^one poly\'edral saillant dual de 
$\sum_{i = 1}^{r}\mathbb{Q}_{\geq 0}\,m_{i}$ et pour tout $z\in C$, consid\'erons 
$\Delta_{z}[f]\subset N_{\mathbb{Q}}$ le $\sigma$-poly\`edre d\'efini par les in\'egalit\'es 
\begin{eqnarray*}
m_{i}\geq -\nu_{z}(f_{i}),\,\,\,\,\,i= 1,2,\ldots,r.
\end{eqnarray*}
On note $\mathfrak{D}[f]$ le diviseur $\sigma$-poly\'edral 
$\sum_{z\in C}\Delta_{z}[f]\,.\,z$.
\end{notation}
Dans le th\'eor\`eme suivant, nous d\'ecrivons la normalisation de 
l'alg\`ebre $A$ donn\'ee dans 2.1.
\begin{theorem}
Soient $C$ une courbe alg\'ebrique lisse et
\begin{eqnarray*}
A = \mathbf{k}[C][f_{1}\chi^{m_{1}},\ldots,f_{r}\chi^{m_{r}}] \subset 
\mathbf{k}(C)[M]
\end{eqnarray*}
une sous-alg\`ebre $\mathbb{T}$-stable engendr\'ee par des \'el\'ements homog\`enes 
$f_{1}\chi^{m_{1}},\dots,f_{r}\chi^{m_{r}}$. 
Supposons que $A$ a le m\^eme corps des fractions que $\mathbf{k}(C)[M]$ 
et soit $\bar{A}$ la normalisation de $A$, 
vue comme sous-alg\`ebre de $\mathbf{k}(C)[M]$. Consid\'erons $\mathfrak{D}[f]$ d\'efini
comme dans $2.3$. Alors 
le dual $\sigma$ du c\^one poly\'edral 
$\sum_{i = 1}^{r}\mathbb{Q}_{\geq 0}m_{i}$ est saillant et $\mathfrak{D}[f]$ est 
l'unique diviseur $\sigma$-poly\'edral propre sur $C$ v\'erifiant $\bar{A} = A[C,\mathfrak{D}[f]]$.   
\end{theorem}
\begin{proof} 
D'apr\`es [HS, Theorem 2.3.2], la sous-alg\`ebre 
$\bar{A}\subset \mathbf{k}(C)[M]$ est 
$M$-gradu\'ee. Puisque $A$ a m\^eme corps des fractions
que $\mathbf{k}(C)[M]$, le c\^one $\sigma$ est saillant. Par le th\'eor\`eme $1.7$, 
il existe un unique diviseur $\sigma$-poly\'edral propre $\mathfrak{D}$ sur $C$ tel que 
$\bar{A} = A[C,\mathfrak{D}]$.
Posons $B := A[C,\mathfrak{D}[f]]$. Alors pour tout $i\in \rm\{1\it ,\ldots,r\rm\}$ et tout $z\in C$,
on a l'in\'egalit\'e
\begin{eqnarray*}
h_{\Delta_{z}[f]}(m_{i})\geq -\nu_{z}(f_{i}),
\end{eqnarray*}
de sorte que pour tout $i\in \rm\{1\it ,\ldots,r\rm\}$,
\begin{eqnarray*}
\mathfrak{D}[f](m_{i}) = \sum_{z\in C}h_{\Delta_{z}[f]}(m_{i})\,.\,z\geq -\rm div(\it f_{i}\rm ).
\end{eqnarray*}
On obtient l'inclusion $A\subset B$. Comme $B$ est l'intersection d'anneaux de valuations discr\`etes
de $\rm Frac\,\it B/\mathbf{k}$ [De, $\S 2.7$], l'alg\`ebre $B$ est normale. 
D'o\`u $\bar{A} = A[\mathfrak{D}]\subset B$. 

Montrons l'\'egalit\'e $\mathfrak{D} = \mathfrak{D}[f]$. Supposons que $C$ est une courbe alg\'ebrique projective lisse de genre $g$. 
Pour tout $m'\in\sigma^{\vee}_{M}$, on a 
\begin{eqnarray}
H^{0}(C,\mathcal{O}_{C}(\lfloor \mathfrak{D}(m')\rfloor ))\subset H^{0}(C,\mathcal{O}_{C}(\lfloor \mathfrak{D}[f](m')\rfloor )).
\end{eqnarray}
Soit $m\in\sigma^{\vee}_{M}$ appartenant \`a l'int\'erieur relatif de $\sigma^{\vee}$. Prenons $s\in\mathbb{Z}_{>0}$ tels que $s\mathfrak{D}(m)$ et $s\mathfrak{D}[f](m)$
soient des diviseurs de Weil \`a coefficients entiers. Alors par $(5)$, par le th\'eor\`eme de Riemann-Roch et puisque $\mathfrak{D}$ est propre, il existe
$r_{0}\in\mathbb{Z}_{>0}$ tel que pour tout entier $r\geq r_{0}$,
\begin{eqnarray*}
h^{0}(C,\mathcal{O}_{C}(rs\mathfrak{D}[f](m)))\geq r\rm deg(\it s\mathfrak{D}(m)\rm) + 1 - \it g \rm > 1.
\end{eqnarray*}
D'apr\`es [Ha, IV, Lemma 1.2], pour tout entier $r\geq r_{0}$, $\rm deg(\it rs\mathfrak{D}[f](m)\rm) > 0$ et donc $\mathfrak{D}[f](m)$ est semi-ample. Donc par 
[AH, Lemma 9.1], on a $\mathfrak{D}[f](m)\geq \mathfrak{D}(m)$. 
Cette in\'egalit\'e est aussi vraie lorsque $C$ est une courbe alg\'ebrique affine lisse. 
Revenons \`a l'hypoth\`ese o\`u $C$ est une courbe alg\'ebrique lisse quelconque. 
Les in\'egalit\'es 
\begin{eqnarray*}
\mathfrak{D}(m_{i}) + \rm div(\it f_{i}\rm )\geq 0,\,\,\,\it i\rm  = 1\it ,\ldots,r,
\end{eqnarray*}
entra\^inent que pour tout $z\in C$, $\Delta_{z}\subset \Delta_{z}[f]$. 
Donc pour tout $m\in\sigma^{\vee}$, $\mathfrak{D}(m)\geq\mathfrak{D}[f](m)$. 
Par le lemme $2.2$, on conclut que $\mathfrak{D} = \mathfrak{D}[f]$.
\end{proof}
L'exemple qui suit illustre comment on peut \'etablir la normalit\'e d'une alg\`ebre donn\'ee
\`a partir d'un syst\`eme de g\'en\'erateurs.
\begin{exemple}
Soit $z$ une variable et soit $\mathbb{T}$ le tore $(\mathbf{k}^{\star})^{2}$. Consid\'erons 
la sous-alg\`ebre $\mathbb{T}$-stable 
\begin{eqnarray*}
A = \mathbf{k}\left[\frac{z}{z-1}\chi^{(2,0)},\chi^{(0,1)},z\chi^{(2,2)}, 
\frac{z^{2}}{z-1}\chi^{(3,2)}\right]\subset \mathbf{k}(z)[\mathbb{Z}^{2}].
\end{eqnarray*}
Alors on a les \'egalit\'es $A^{\mathbb{T}} = \mathbf{k}$ et 
$(\rm Frac\it\,A)^{\mathbb{T}} = \mathbf{k}(z)$. Le c\^one des poids $\sigma^{\vee}\subset\mathbb{Q}^{2}$ 
de $A$ est le quadrant positif. Par le th\'eor\`eme $2.4$, la normalisation $\bar{A}$ de $A$ est \'egale \`a 
$A[\mathbb{P}^{1},\mathfrak{D}]$ o\`u $\mathfrak{D} = \Delta_{0}(0) + \Delta_{1}(1) + \Delta_{\infty}(\infty)$ est 
le diviseur $\sigma$-poly\'edral propre sur $\mathbb{P}^{1}$ d\'efini par 
\begin{eqnarray*}
\Delta_{0} = -\left(\frac{1}{2},0\right)+\sigma,\,\,
\Delta_{1} = \left(\frac{1}{2},0\right)+\sigma,\,\,
\Delta_{\infty} = \left[\left(\frac{1}{2},0),(0,\frac{1}{2}\right)\right]+\sigma.
\end{eqnarray*}
Par exemple, le poly\`edre $\Delta_{0}$ est donn\'e par les in\'egalit\'es
\begin{eqnarray*}
2x\geq -\nu_{0}\left(\frac{z}{z-1}\right) = -1,\,\,\, y\geq -\nu_{0}(1) = 0,\,\,\,  2x+ 2y\geq -\nu_{0}(z) = -1, 
\end{eqnarray*}
\begin{eqnarray*}
3x +2y \geq -\nu_{0}\left(\frac{z^{2}}{z- 1}\right) = -2\,.
\end{eqnarray*}
Un calcul direct montre qu'en fait $A = A[\mathbb{P}^{1},\mathfrak{D}]$. Si l'on pose 
\begin{eqnarray*}
t_{1}:= \frac{z}{z-1}\chi^{(2,0)},\,\,\,t_{2} := \chi^{(0,1)},\,\,\, t_{3} := z\chi^{(2,2)},\,\,\,
t_{4} := \frac{z^{2}}{z-1}\chi^{(3,2)},
\end{eqnarray*}
alors les fonctions $t_{1},t_{2},t_{3},t_{4}$ v\'erifient la relation irr\'eductible
$t_{4}^{2} - t_{1}^{2}t_{2}^{2}t_{3} - t_{1}t_{3}^{2} = 0$. On conclut que 
l'hypersurface $V(x_{4}^{2} - x_{1}^{2}x_{2}^{2}x_{3} - x_{1}x_{3}^{2})\subset \mathbb{A}^{4}$ est normale.
\end{exemple}
Rappelons que pour une vari\'et\'e alg\'ebrique affine $X$, on note $\rm SAut\,\it X$ le sous-groupe des automorphismes 
de $X$ engendr\'e par la r\'eunion des images des op\'erations alg\'ebriques de $\mathbb{G}_{a}$ dans $X$. 
Pour plus d'informations voir [AKZ] et [AFKKZ] o\`u la notion de vari\'et\'e flexible est introduite. 
L'exemple suivant est inspir\'e de 
[LS, Example 1.1]. Il correspond au cas de $d = 3$ et de $e = 2$.
\begin{exemple}
Soient $d,e\geq 2$ des entiers tels que $d + 1 = e^{2}$. Notons
\begin{eqnarray*}
\mathfrak{D}_{d,e} = \Delta_{0}^{d,e}\cdot\{0\} + \Delta_{1}^{d,e}\cdot\{1\} + \Delta_{\infty}^{d,e}\cdot\{\infty\} 
\end{eqnarray*}
le diviseur poly\'edral propre sur $\mathbb{P}^{1}$ avec 
\begin{eqnarray*}
\Delta_{0}^{d,e} = [(1,0),(1,1)] + \sigma_{de},\,\, \Delta_{1}^{d,e} = \left(-\frac{1}{e},0\right) + \sigma_{de},\,\,
\Delta_{\infty}^{d,e} = \left(\frac{e}{d} - 1,0\right) + \sigma_{de},
\end{eqnarray*}
o\`u $\sigma_{de}$ est le c\^one de $\mathbb{Q}^{2}$ engendr\'e par les vecteurs $(1,0)$ et $(1,de)$, 
et $X_{d,e} := X[\mathfrak{D}_{d,e}]$. Alors $X_{d,e}$ n'est pas torique. 
En effet, $X_{d,e}$ n'est pas $\mathbb{T}$-isomorphe \`a $X[\mathbb{P}^{1},\mathfrak{D}]$ 
avec $\mathfrak{D}$ un diviseur poly\'edral propre 
sur $\mathbb{P}^{1}$ port\'e par au plus $2$ points, i.e. $\rm Card\,Supp\,\it\mathfrak{D}\rm\leq 2$ 
(voir [AL, Corollary 1.4] et le th\'eor\`eme $1.9$). 

Montrons que le groupe $\rm SAut\,\it X_{d,e}$ 
op\`ere infiniment transitivement dans le lieu lisse de $X_{d,e}$. Pour cela consid\'erons la
sous-alg\`ebre
\begin{eqnarray*}
A_{d,e} :=  \mathbf{k}\left[\chi^{(0,1)},\,\frac{(1-z)^{d}}{z^{de - 1}}\chi^{(de,-1)},\,\frac{1-z}{z^{e}}\chi^{(e,0)},\, 
\frac{(1-z)^{e}}{z^{d}}\chi^{(d,0)}\right]\subset \mathbf{k}(z)[\mathbb{Z}^{2}] 
\end{eqnarray*}
o\`u $\mathbb{T}$ est le tore $(\mathbf{k}^{\star})^{2}$ et $z$ est une variable sur $\mathbf{k}(\mathbb{T})$. 
Posons 
\begin{eqnarray*}
 u := \chi^{(0,1)},\,\,\, v := \frac{(1-z)^{d}}{z^{de - 1}}\chi^{(de,-1)},\,\,\, x := \frac{1-z}{z^{e}}\chi^{(e,0)},\,\,\, 
y := \frac{(1-z)^{e}}{z^{d}}\chi^{(d,0)}.
\end{eqnarray*}
Alors les fonctions $u,v,x,y$ v\'erifient la relation irr\'eductible $uv + x^{d} - y^{e} = 0 \,\,\, (E)$,
identifiant $\rm Spec\,\it A_{d,e}$ avec l'hypersurface $H_{d,e}$ d'\'equation $(E)$ de $\mathbb{A}^{4}$. 
L'origine $O$ de $\mathbb{A}^{4}$ est l'unique point singulier de $H_{d,e}$. 

D'apr\`es [Vie, Proposition 2], $H_{d,e}$ est normale en $O$ 
et par [KZ, \S 5], [AKZ, Theorem 3.2], le groupe $\rm SAut\,\it H_{d,e}$ op\`ere
infiniment transitivement dans $H_{d,e}-\{O\}$ comme suspension du plan affine. Par ailleurs, 
on a $(\rm Frac\,\it A_{d,e})^{\mathbb{T}} = \mathbf{k}(z)$ et 
$\rm Frac\,\it A_{d,e} = \rm Frac\it\, \mathbf{k}(z)[\mathbb{Z}^{\rm 2}]$.
Un calcul facile montre que $\mathfrak{D}[u,v,x,y] = \mathfrak{D}_{d,e}$. On conclut par le th\'eor\`eme $2.4$.

Notons que le calcul de $\mathfrak{D}_{d,e}$ peut \^etre obtenu \`a partir de $H_{d,e}$ en utilisant [AH, \S 11].
\end{exemple}
Le th\'eor\`eme $2.4$ appliqu\'e \`a $\mathbb{T} = \mathbf{k}^{\star}$ donne le corollaire suivant. 
Les parties concernant les cas paraboliques et hyperboliques
ont \'et\'e \'etablies dans [FZ, 3.9, 4.6]. Gr\^ace \`a ces r\'esultats, il est montr\'e dans [FZ]
que toute $\mathbb{C}^{\star}$-surface complexe normale affine parabolique ayant un bon quotient
isomorphe \`a $\mathbb{A}^{1}_{\mathbb{C}}$ est la normalisation d'une hypersurface de $\mathbb{A}^{3}_{\mathbb{C}}$
d'\'equation $x^{d} = yP(z)$, o\`u $P\in\mathbb{C}[z]$ est un polyn\^ome. Une description analogue est donn\'ee
dans le cas hyperbolique, c.f. [FZ, $4.8$]. 
\begin{corollaire}
Soit $C$ une courbe alg\'ebrique lisse, consid\'erons $\chi$ une variable sur $\mathbf{k}(C)$ et soit $A$ une sous-alg\`ebre
gradu\'ee de $\mathbf{k}(C)[\mathbb{Z}]$ ayant le 
m\^eme corps des fractions
que $\mathbf{k}(C)[\mathbb{Z}]$.
Alors les assertions suivantes sont vraies. 
\begin{enumerate}
\item[\rm (i)]\rm (cas elliptique et parabolique) \em
Si 
\begin{eqnarray*}
A = \mathbf{k}[C][f_{1}\chi^{m_{1}},\dots,f_{r}\chi^{m_{r}}]\subset \mathbf{k}(C)[\mathbb{Z}],
\end{eqnarray*} 
avec $f_{i}\in \mathbf{k}(C)^{\star}$ et $m_{1},\ldots,m_{r}\in\mathbb{Z}_{>0}$, alors la pr\'esentation de
Dolgachev-Pinkham-Demazure (D.P.D.) de $\bar{A} = A_{C,D}$ est donn\'ee par le $\mathbb{Q}$-diviseur sur $C$
\begin{eqnarray*}
D = - \min_{1\leq i\leq r}\frac{\rm div(\it f_{i}\rm)}{\it m_{i}}\,;
\end{eqnarray*}
\item[\rm (ii)]\rm (cas hyperbolique) \em 
Si $C=\rm Spec\it\, A_{\rm 0}$ est une courbe affine lisse, 
\begin{eqnarray*}
A = \mathbf{k}[C][f_{1}\chi^{-m_{1}},\ldots,f_{r}\chi^{-m_{r}},g_{1}\chi^{n_{1}},\ldots,g_{s}\chi^{n_{s}} ]\subset 
\mathbf{k}(C)[\mathbb{Z}],
\end{eqnarray*}
avec $n_{1},\ldots,n_{s}, m_{1},\ldots,m_{r}\in\mathbb{Z}_{>0}$, alors la pr\'esentation $D.P.D.$ 
de $\bar{A} = A_{0}[D_{-},D_{+}]$ est donn\'ee par les $\mathbb{Q}$-diviseurs 
\begin{eqnarray*}
D_{-} = - \min_{1\leq i\leq r}\frac{\rm div(\it f_{i}\rm)}{\it m_{i}}\,\,\,\rm et\,\,\,\it D_{+} = - \min_{\rm 1\it\leq i\rm\leq\it s}\frac{\rm div(\it g_{i}\rm)}{\it n_{i}}\,. 
\end{eqnarray*}
\end{enumerate}
\end{corollaire}
\begin{proof}
Nous donnons la preuve pour le cas de $\rm (i)$. L'assertion $\rm (ii)$ se traite de la
m\^eme fa\c con. D'apr\`es le th\'eor\`eme $2.4$, on peut \'ecrire $\bar{A} = A[C,\mathfrak{D}]$ avec
$\mathfrak{D} = \sum_{z\in C}\Delta_{z}\cdot z$
o\`u pour $z\in C$,
\begin{eqnarray*}
\Delta_{z} = \left\{v\in\mathbb{Q},\,m_{i}\cdot v\geq -\nu(f_{i}),\,1\leq i\leq r\right\}. 
\end{eqnarray*}
Puisque $D = \mathfrak{D}(1)$, nous obtenons le r\'esultat.
\end{proof}

\section{Id\'eaux int\'egralement clos et poly\`edres entiers} 
Dans cette section, nous rappelons une description des id\'eaux monomiaux int\'egralement
clos en terme de poly\`edres entiers. Nous commen\c cons par quelques notions g\'en\'erales.
 \begin{rappel}
Puisque nous restons dans un contexte g\'eom\'etrique, la lettre $A$ d\'esigne une alg\`ebre int\`egre de 
type fini sur $\mathbf{k}$. Soit $I\subset A$ un id\'eal. Un \'el\'ement $a\in A$ est dit \em entier \rm (ou satisfaisant
une relation de d\'ependance int\'egrale) sur $I$ s'il existe $r\in\mathbb{Z}_{>0}$ et des \'el\'ements 
\begin{eqnarray*}
\lambda_{1}\in I,\,\lambda_{2}\in I^{2},\ldots,\,\lambda_{r}\in I^{r}\,\,\,\,\,\rm tels\,\,que\,\,\,\,\it a^{r}+\sum_{i \rm = 1}^{r}
\lambda_{i}a^{r\rm -\it i} \rm = 0\,.
\end{eqnarray*}
L'id\'eal $I$ est dit \em int\'egralement clos \rm (ou complet) si tout entier de $A$ sur $I$ appartient \`a $I$. 
On dit que $I$ est  \em normal \rm si pour tout
entier $i\geq 1$, l'id\'eal $I^{i}$ est int\'egralement clos. 

Le sous-ensemble $\bar{I}\subset A$, appel\'e \em cl\^oture int\'egrale \rm (ou fermeture int\'egrale)
 de $I$ dans $A$, est l'ensemble des \'el\'ements entiers sur $I$. 
C'est le plus petit id\'eal int\'egralement clos de $A$ contenant $I$ [HS, Corollary 1.3.1].
La d\'emonstration de ce fait utilise les r\'eductions d'id\'eaux (voir [NR]). 
Pour plus d'informations sur les 
id\'eaux int\'egralement clos,
voir par exemple [LeTe], [HS], [Va].
Consid\'erons l'alg\`ebre de Rees 
\begin{eqnarray*}
B = A[It] = A\oplus\bigoplus_{i\geq 1}I^{i}t^{i}
\end{eqnarray*}
correspondante \`a un id\'eal $I$ de $A$. Alors la normalisation de $B$ est  
\begin{eqnarray}
 \bar{B} = \bar{A}\oplus\bigoplus_{i\geq 1}\overline{\bar{A}\,\, I^{i}}t^{i}
\end{eqnarray}
o\`u $\bar{A}$ est la normalisation de $A$ (voir [Ri]). La normalisation de $B$ 
est \'egale \`a celle de $\bar{A}[\bar{A}It]$ [HS, Proposition 5.2.4]. 
Si $A$ est normale alors $A[It]$ est normale si et seulement si $I$ est normal. 

Supposons maintenant que $A$ est 
$M$-gradu\'ee normale. Un id\'eal $I$ de $A$ est dit \em homog\`ene \rm si $I$ est non nul et si $I$ est engendr\'e par des
\'el\'ements homog\`enes de $A$. Chaque id\'eal homog\`ene de $A$ est un sous-$\mathbb{T}$-module rationnel [Br, \S 1.1] et
admet donc une d\'ecomposition en somme directe de sous-espaces propres. D'apr\`es [HS, Corollary 5.2.3], 
si $I$ est homog\`ene alors $\bar{I}$ est homog\`ene.
\end{rappel}
En l'absence de r\'ef\'erence, nous donnons la preuve du lemme suivant.
\begin{lemme}
Soit $A$ une alg\`ebre normale de type fini sur $\mathbf{k}$ et soit $I$ un id\'eal de $A$. Alors 
la fermeture int\'egrale de $A[It]$ dans son corps des fractions est
\'egale \`a celle de $A[\bar{I}t]$. En particulier, pour tout $i\in\mathbb{Z}_{>0}$, 
$\overline{I^{i}} = \overline{\bar{I}^{i}}$.  
\end{lemme}
\begin{proof} 
Pour tout $i\in\mathbb{Z}_{>0}$, on a  $I^{i}\subset \bar{I}^{i}$, soit $A[It]\subset A[\bar{I}t]$.
D'o\`u $\overline{A[It]}\subset \overline{A[\bar{I}t]}$. Par $(6)$ et puisque $A = \bar{A}$ est normale, 
$\bar{I}t$ est inclus dans $\overline{A[It]}$, donc 
$A[\bar{I}t]\subset\overline{A[It]}$ et finalement  $\overline{A[\bar{I}t]}\subset\overline{A[It]}$. La deuxi\`eme affirmation
est une cons\'equence de $(6)$ et de l'\'egalit\'e $\overline{A[It]} = \overline{A[\bar{I}t]}$.
\end{proof}  
\begin{rappel} 
Fixons un c\^one poly\'edral saillant $\sigma\subset N_{\mathbb{Q}}$. 
Consid\'erons une partie $F\subset M_{\mathbb{Q}}$ telle que $F\cap\sigma^{\vee}\neq \emptyset$.
On note $\rm Pol_{\sigma^{\vee}}(\it F)$ l'ensemble
des poly\`edres de la forme $P = Q + \sigma^{\vee}$ avec $Q$ un polytope dont
les sommets appartiennent \`a $F$. Un \em $\sigma^{\vee}$-poly\'edre entier \rm
est un \'el\'ement de $\rm Pol_{\sigma^{\vee}}(\it M)$.

Si $P\in\rm Pol_{\sigma^{\vee}}(\it M)$ alors 
pour tout entier $e\geq 1$, on pose $eP := P+\ldots +P$ la somme de Minkowski de $e$ exemplaires de $P$. 
Le poly\`edre $eP$ est l'image de $P$ par l'homoth\'etie de centre $0$ et de rapport $e$. Si $e = 0$ 
alors on pose $eP = \sigma^{\vee}$. On dit que $P$ est \em normal \rm si pour tout entier $e\geq 1$,
\begin{eqnarray*}
(eP)\cap M = \{m_{1}+\ldots+m_{e}\,|\,m_{1},\ldots,m_{e}\in P\cap M\}\,.        
\end{eqnarray*}
Un sous-semi-groupe $E\subset M$ est dit \em satur\'e \rm si pour
$d\in\mathbb{Z}_{>0}$ et pour $m\in M$ tels que $dm\in E$, on a $m\in E$.
Cela est \'equivalent \`a ce que $E$ soit l'intersection de $\rm Cone(\it E)$ et du r\'eseau $M$.
\end{rappel}
L'assertion suivante est ais\'ee et nous laissons la d\'emonstration aux lecteurs.
\begin{lemme}
Pour un c\^one poly\'edral saillant $\sigma\subset N_{\mathbb{Q}}$ et un poly\`edre entier $P\in\rm Pol_{\sigma^{\vee}}(\it M)$, on note 
\begin{eqnarray*}
S = S_{P} := \{(m,e)\in M\times\mathbb{N}\,|\,m\in (eP)\cap M\}\,.
\end{eqnarray*}
Alors $S$ est un sous-semi-groupe satur\'e de $M\times\mathbb{Z}$. De plus, pour tout $e\geq 1$, l'enveloppe convexe de  
\begin{eqnarray*} 
E_{[e,P]} := \{m_{1}+\ldots +m_{e}\,|\,m_{1},\ldots,m_{e}\in P\cap M\}
\end{eqnarray*}
dans $M_{\mathbb{Q}}$ est \'egale \`a $eP$.
\end{lemme}  

Rappelons que tout id\'eal homog\`ene int\'egralement clos d'une vari\'et\'e torique affine est caract\'eris\'e 
par l'enveloppe convexe de l'ensemble de ses poids. Voir [Vit, $3.1$] pour le cas de l'alg\`ebre 
des polyn\^omes \`a plusieurs variables. Le r\'esultat suivant est connu (voir 
[CLS, Proposition $11.3.4$] pour le cas o\`u $\sigma^{\vee}$ est saillant, ainsi\footnote{
Nous remercions un des rappoteurs de nous avoir communiqu\'e cette r\'ef\'erence.} que [KKMS, Chapter I, \S$2$] pour
le cas g\'en\'eral). Par commodit\'e,
nous donnons une courte preuve. Cela nous sera utile pour la suite. 
Pour tout entier $e$, nous noterons $M_{e} = M\times\{e\}$.

\begin{theorem}
Soit $\sigma\subset N_{\mathbb{Q}}$ un c\^one poly\'edral saillant. Alors l'application 
\begin{eqnarray*}
P\mapsto I[P] = \bigoplus_{m\in P\cap M}\mathbf{k}\,\chi^{m}
\end{eqnarray*}
est une bijection entre $\rm Pol_{\sigma^{\vee}}(\it \sigma^{\vee}_{M})$ 
et l'ensemble des id\'eaux homog\`enes int\'egralement clos de $A:= \mathbf{k}[\sigma^{\vee}_{M}]$. 
Plus pr\'ecis\'ement, 
si $I = (\chi^{m_{1}},\ldots,\chi^{m_{r}})$ est un id\'eal homog\`ene de $A$ 
alors sa fermeture $\bar{I}$ est $I[P]$ o\`u 
\begin{eqnarray}
P = \rm Conv(\it m_{\rm 1},\ldots,m_{r}\rm)+\sigma^{\vee}.
\end{eqnarray}
Le c\^one des poids $\widetilde{\omega}$ de $A[It]$ est
l'unique c\^one v\'erifiant $\widetilde{\omega}\cap M_{e}
 = (eP)\cap M$, pour tout $e\in\mathbb{N}$.
De plus, si $P\in\rm Pol_{\sigma^{\vee}}(\it \sigma^{\vee}_{M})$
alors $P$ est normal si et seulement si $I[P]$ est normal.
\end{theorem}
\begin{proof}
La justification de la bijectivit\'e de $I\mapsto I[P]$ provient de
[KKMS, Chapter I, \S$2$] et nous l'omettons. 

Soit $I\subset A$
un id\'eal engendr\'e par $\chi^{m_{1}},\ldots,\chi^{m_{r}}$. 
Consid\'erons $P$ comme dans $(7)$. Montrons
l'\'egalit\'e $\bar{I} = I[P]$. 
Puisque $I\subset I[P]$ et comme $I[P]$ est int\'egralement clos, on
a $\bar{I}\subset I[P]$. Soit $\chi^{m}\in I[P]$. Alors $m\in P\cap M$.
Donc il existe $c_{1},\ldots, c_{r}\in \mathbb{Q}_{\geq 0}$ de somme
\'egale \`a $1$ et $c\in\sigma^{\vee}$ tels que 
\begin{eqnarray*}
 m = \sum_{i = 1}^{r}c_{i}m_{i} + c. 
\end{eqnarray*}
Soit $d'\in\mathbb{Z}_{>0}$ tel que $d'c_{i}\in \mathbb{N}$ pour chaque $i$ et
tel que $dc\in\sigma^{\vee}_{M}$. Alors $\chi^{d'm}\in I^{d'}$ et donc
$\chi^{m}\in \bar{I}$. D'o\`u $\bar{I} = I[P]$ et la surjectivit\'e de $\varphi$.

D\'eterminons le c\^one des poids $\widetilde{\omega}$ de $A[It]$. D'apr\`es le lemme $3.2$,
on peut supposer que $I = I[P]$.
Soit $e\in\mathbb{Z}_{>0}$. Alors on a 
\begin{eqnarray*}
I[P]^{e} = \bigoplus_{m\in E[e,P]}\mathbf{k}\chi^{m}. 
\end{eqnarray*}
Notons $\overline{I[P]^{e}} = I[P_{e}]$, pour un poly\`edre 
$P_{e}\in\rm Pol_{\sigma^{\vee}}(\it \sigma^{\vee}_{M})$. Par le lemme $3.4$, 
on a d'une part,
\begin{eqnarray*}
eP = \rm Conv(\it E_{[e,P]}) \subset P_{e}. 
\end{eqnarray*}
D'autre part, $I[P]^{e}\subset I[eP]$ et donc $I[P_{e}]\subset I[eP]$.
D'o\`u $eP = P_{e}$. D'apr\`es l'\'egalit\'e $(6)$ de $3.1$, $\widetilde{\omega}$ est le c\^one des
poids de
\begin{eqnarray*}
\overline{A[It]} = A\oplus\bigoplus_{e\in\mathbb{Z}_{>0}}I[eP]t^{e}. 
\end{eqnarray*} 
D'o\`u $\widetilde{\omega}\cap M_{e} = (eP)\cap M$, pour tout entier $e\geq 0$.
Les autres assertions suivent ais\'ement.
\end{proof}

\paragraph{}
Le th\'eor\`eme suivant est une cons\'equence d'un r\'esultat
sur la normalit\'e des polytopes entiers [BGN, Theorem $1.1.3$]. 
\begin{theorem}
Soit $\sigma\subset N_{\mathbb{Q}}$ un c\^one poly\'edral saillant, notons 
$n = \rm dim\,\it N_{\mathbb{Q}}\rm $ et consid\'erons $P\in\rm Pol_{\sigma^{\vee}}(\it M)$. 
Alors pour tout entier $e\geq n-1$, le poly\`edre $eP$ est normal. 
En particulier, tout id\'eal stable int\'egralement clos d'une surface torique 
est normal.    
\end{theorem}
\begin{proof}
Soit $\mathscr{E}$ une subdivision de c\^ones poly\'edraux saillants de $\sigma^{\vee}$
de dimension $n$. Ecrivons $P = Q + \sigma^{\vee}$ avec $Q$ un polytope entier. 
Soit $m\in (eP)\cap M$. Alors il existe $\tau\in\mathscr{E}$ tel que $m\in (eQ + \tau)\cap M$.
D'apr\`es [CLS, Proposition $7.1.9$], on a $m\in E_{[e,Q+\tau]}$ (voir $3.4$). Donc $m\in E_{[e,P]}$ et
$eP$ est normal. Le reste provient du th\'eor\`eme $3.5$.  
\end{proof}
 
\begin{remarque}
Notons que dans le cas o\`u $\mathbf{k}[\sigma^{\vee}_{M}] = 
\mathbf{k}[x_{1},\dots,x_{n}]$ est l'alg\`ebre des polyn\^omes, un id\'eal monomial 
$I\subset\mathbf{k}[x_{1},\ldots,x_{n}]$ est normal si et seulement si pour tout $i= 1,\ldots,n-1$, 
l'id\'eal $I^{i}$ est int\'egralement clos [RRV, $3.1$]. Nous verrons une g\'en\'eralisation de
ce r\'esultat dans la section $5$.

D'apr\`es [ZS, Appendix $5$], [HS, $\S 1.1.4$, $\S 14.4.4$], 
tout id\'eal int\'egralement clos d'une surface affine lisse est normal. Cependant cela ne s'applique 
pas aux vari\'et\'es toriques affines de dimension $3$. 
En effet, d'apr\`es [HS, Exercice 1.14], 
si $\sigma\subset\mathbb{Q}^{3}$ est l'octant positif et si
\begin{eqnarray*}
P := \rm Conv((2,0,0),\, (0,3,0),\,(0,0,7)) + \it\sigma^{\vee}  
\end{eqnarray*} 
alors $I[P]$ n'est pas normal.
\end{remarque}

\section{Id\'eaux int\'egralement clos et $\mathbb{T}$-vari\'et\'es affines de complexit\'e un} 
Dans cette section, $C$ d\'esigne une courbe alg\'ebrique lisse, le sous-ensemble 
$\sigma\subset N_{\mathbb{Q}}$ est un c\^one poly\'edral saillant et $\mathfrak{D} = \sum_{z\in C}\Delta_{z}\cdot z$
 est un diviseur $\sigma$-poly\'edral propre. Posons $A := A[C, \mathfrak{D}]$. 
Nous allons \'etudier
les id\'eaux homog\`enes int\'egralement clos
de $A$. Nous commen\c cons par d\'ecrire le c\^one des poids de l'alg\`ebre de 
Rees de $A$ selon un id\'eal homog\`ene.
Rappelons que pour un entier $e$, $M_{e}$ d\'esigne $M\times\{e\}$.
\begin{lemme}
Soit $I\subset A$ un id\'eal engendr\'e 
par des \'el\'ements homog\`enes $f_{1}\chi^{m_{1}},\ldots, f_{r}\chi^{m_{r}}$. Notons
\begin{eqnarray*}
P = \rm Conv(\it m_{\rm 1\it},\ldots m_{r})+\sigma^{\vee}.
\end{eqnarray*}
Alors le c\^one des poids $\widetilde{\omega}$ 
de l'alg\`ebre $(M\times\mathbb{Z})$-gradu\'ee $A[It]$ v\'erifie $\widetilde{\omega}\cap M_{e} = 
(eP)\cap M$, pour tout $e\in\mathbb{N}$.
\end{lemme}
\begin{proof}
Soit $E$ l'ensemble des
poids de $I$. La partie
$J:=\bigoplus_{m\in E}\mathbf{k}\chi^{m}$ 
est un id\'eal de $A':=\mathbf{k}[S]$.
Par le lemme $3.2$, on a
\begin{eqnarray*}
 \overline{A'[Jt]} = \overline{B[\overline{JB}t]}\rm\,\,\, avec\,\,\, \it B:=\mathbf{k}[\sigma^{\vee}_{M}].
\end{eqnarray*}
Les \'el\'ements 
$\chi^{m_{1}},\ldots, \chi^{m_{r}}$ engendrent l'id\'eal $JB$. 
Puisque $\widetilde{\omega}$ est le c\^one des poids de $\overline{A'[Jt]}$,
le reste provient du th\'eor\`eme $3.5$. 
\end{proof}
L'assertion suivante permet de calculer explicitement la fermeture int\'egrale des
id\'eaux homog\`enes de l'alg\`ebre $A$.
\begin{theorem}
Soit $I$ un id\'eal de $A = A[C,\mathfrak{D}]$ engendr\'e par des \'el\'ements homog\`enes
$f_{1}\chi^{m_{1}},\ldots,f_{r}\chi^{m_{r}}$. 
Alors on a 
\begin{eqnarray*}
 A = \bigoplus_{m\in\sigma^{\vee}_{M}}
H^{\rm 0\it }(C,\mathcal{O}_{C}(\lfloor\widetilde{\mathfrak{D}}(m,\rm 0\it)\rfloor ))\chi^{m}
\end{eqnarray*}
et pour tout entier $e\geq 1$,
\begin{eqnarray*}
\overline{I^{e}} =  \bigoplus_{m\in (eP)\cap M}
H^{0}(C,\mathcal{O}_{C}(\lfloor\widetilde{\mathfrak{D}}(m,e)\rfloor ))\chi^{m}
\end{eqnarray*} 
o\`u $(P,\widetilde{\mathfrak{D}})$ v\'erifie les conditions suivantes.
\begin{enumerate}
 \item[\rm (i)] $P$ est le poly\`edre $\rm Conv(\it m_{\rm 1},\ldots,m_{r}\rm)+\it\sigma^{\vee}$;
\item[\rm (ii)] $\widetilde{\mathfrak{D}}$ est le diviseur poly\'edral sur $C$
dont ses coefficients sont d\'efinis par
\begin{eqnarray*}
\widetilde{\Delta}_{z} = \bigcap_{i = 1}^{r}\left\{\,(v,p)\in N_{\mathbb{Q}}\times\mathbb{Q},\,
m_{i}(v) + p \geq -\nu_{z}(f_{i})\,\right\}\cap\left(\Delta_{z}\times\mathbb{Q}\right), 
\end{eqnarray*}
pour tout $z\in C$.
\end{enumerate}
\end{theorem}
\begin{proof}
On applique le th\'eor\`eme  $2.4$ \`a l'alg\`ebre $A[It]$
en utilisant les \'el\'ements $f_{i}\chi^{m_{i}}$ et 
des g\'en\'erateurs homog\`enes de $A$. Le diviseur
poly\'edral correspondant $\widetilde{\mathfrak{D}}$ v\'erifie
$\rm (ii)$. On conclut par le lemme $4.1$. 
\end{proof}

\begin{exemple}
Reprenons l'exemple $2.5$. On consid\`ere l'id\'eal homog\`ene
\begin{eqnarray*}
I = (t_{2},t_{3},t_{4})\subset A = k[t_{1},t_{2},t_{3},t_{4}]\cong 
\frac{k[x_{1},x_{2},x_{3},x_{4}]}{(x_{4}^{2}-x_{1}^{2}x_{2}^{2}x_{3} - x_{1}x_{3}^{2})}\,. 
\end{eqnarray*}
Soit $\widetilde{\omega}\subset \mathbb{Q}^{3}$ le c\^one v\'erifiant 
$\widetilde{\omega}\cap \mathbb{Z}^{2}_{e} = (0,e) + \mathbb{Q}_{\geq 0}^{2}$,
pour tout entier $e\geq 0$. Soit $\widetilde{\mathfrak{D}}$ le diviseur $\widetilde{\omega}^{\vee}$-poly\'edral sur
$\mathbb{P}^{1}$ construit par les g\'en\'erateurs $t_{2},t_{3},t_{4}$. 
Les coefficients non triviaux de $\widetilde{\mathfrak{D}}$ sont
\begin{eqnarray*}
 \widetilde{\Delta}_{0} = -\left(\frac{1}{2},0,0\right) + \widetilde{\omega}^{\vee},\,\, 
\widetilde{\Delta}_{1} = \left(\frac{1}{2},0,0\right) + \widetilde{\omega}^{\vee},\,\,
\widetilde{\Delta}_{\infty} = \rm Conv\left( (0,1,-1),
\left(\frac{1}{2},0,0\right),\left(0,\frac{1}{2},0\right)\right) + \it \widetilde{\omega}^{\vee}.
\end{eqnarray*}
\end{exemple}
Fixons  un poly\`edre  $P\in\rm Pol_{\it\sigma^{\vee}}(\it \sigma^{\vee}_{M})$.
Consid\'erons un c\^one $\widetilde{\omega}\subset M_{\mathbb{Q}}\times\mathbb{Q}$ 
satisfaisant la relation
\begin{eqnarray*}
\widetilde{\omega}\cap M_{e} = (eP)\cap M, 
\end{eqnarray*}
pour tout $e\in\mathbb{N}$. Dans les deux prochains lemmes, nous
\'etudions des diviseurs $\widetilde{\omega}^{\vee}$-poly\'edraux dont les
pi\`eces gradu\'ees correspondantes \`a $M_{e}$ forment un id\'eal 
de $A$.

\begin{lemme}
Soit $\widetilde{\mathfrak{D}} = \sum_{z\in C}\widetilde{\Delta}_{z}\cdot z$ un 
diviseur $\widetilde{\omega}^{\vee}$-poly\'edral propre.
Alors les conditions suivantes sont \'equivalentes.
\begin{enumerate}
 \item[\rm (i)] 
Pour tout 
$z\in C$, le poly\`edre $\Delta_{z}$ est la projection orthogonale de 
$\widetilde{\Delta}_{z}$ parall\`element \`a $\mathbb{Q}\,(0_{M},1)$
et si $(v,p)$ est un sommet de $\widetilde{\Delta}_{z}$ alors $p\leq 0$;
\item[\rm (ii)] Pour tout entier $e\geq 0$, 
\begin{eqnarray*}
 I_{[e]} := \bigoplus_{m\in (eP)\cap M}
H^{ 0}(C,\mathcal{O}_{C}(\lfloor \widetilde{\mathfrak{D}}(m, e )\rfloor ))\chi^{m}
\end{eqnarray*}
est un id\'eal de $A$ et $I_{[0]} = A$.
\end{enumerate}
\end{lemme}
\begin{proof}
Soit $z\in C$. Notons $\Delta'_{z}$ la projection orthogonale de $\widetilde{\Delta}_{z}$
parall\`element \`a l'axe $\mathbb{Q}\,(0_{M},1)$. Montrons l'implication $(\rm i)\Rightarrow (\rm ii)$.
Puisque pour $m\in\sigma^{\vee}$,
$h_{\widetilde{\Delta}_{z}}(m,0) = h_{\Delta'_{z}}(m)$,
on a l'\'egalit\'e $I_{[\rm 0\it ]} = A$. Le lemme $1.10$ implique
que pour tout $e\in\mathbb{N}$, $I_{[e]}\subset A$. D'o\`u
l'assertion $(\rm ii)$.

R\'eciproquement, la propret\'e de $\mathfrak{D}$ et l'\'egalit\'e $A = I_{[0]}$ montre que
$\Delta_{z} = \Delta'_{z}$.
Soit $(v,p)$ un sommet de $\widetilde{\Delta}_{z}$. Il reste \`a montrer que $p\leq 0$.
La partie   
\begin{eqnarray*}
\lambda := \left\{\,(m,e)\in M_{\mathbb{Q}}\times\mathbb{Q},\,h_{\widetilde{\Delta}_{z}}(m,e) = m(v)+ep\,\right\} 
\end{eqnarray*}
est un c\^one d'interieur non vide [AH, $\S 1$]. 
Donc l'ensemble $\lambda_{M\times\mathbb{Z}} = \lambda\cap (M\times\mathbb{Z})$ contient
un \'el\'ement $(m',e')$ v\'erifiant $e'\geq 1$. Soit 
$v'\in \Delta_{z}$ tels que 
\begin{eqnarray*}
 m'(v') = h_{\Delta_{z}}(m').
\end{eqnarray*}
Par propret\'e de $\widetilde{\mathfrak{D}}$ et [AH, Lemma $9.1$], il s'ensuit que 
\begin{eqnarray*}
 m'(v) + e'p = h_{\widetilde{\Delta}_{z}}(m',e')\leq h_{\Delta_{z}}(m') = m'(v'). 
\end{eqnarray*}
Comme $v$ appartient \`a $\Delta_{z}$, on a $m'(v')\leq m'(v)$ et donc $p\leq 0$. D'o\`u le r\'esultat.
\end{proof}
\begin{lemme}
Soit $\widetilde{\mathfrak{D}}$ un diviseur $\widetilde{\omega}^{\vee}$-poly\'edral sur $C$.
Supposons que $\widetilde{\mathfrak{D}}$ v\'erifie la condition
$\rm (ii)$ du lemme $4.4$. Alors pour tout 
$e\in\mathbb{N}$, $I_{[e]}$ est un id\'eal homog\`ene int\'egralement clos.
\end{lemme}
\begin{proof}
Soit $e\in\mathbb{N}$. 
Il suffit de montrer que tout \'el\'ement homog\`ene 
de $\bar{I}_{[e]}$ appartient \`a $I_{[e]}$. 
Soit $a\in\bar{I}_{[e]}$ un \'el\'ement homog\`ene. 
L'\'el\'ement $a\chi^{(0,e)}$ appartient \`a la normalisation de
$A[C,\widetilde{\mathfrak{D}}]$. Puisque $A[C,\widetilde{\mathfrak{D}}]$ est normal [De, $\S 2.7$], 
nous avons $a\in I_{[e]}$.  
\end{proof}
Le th\'eor\`eme suivant d\'ecrit les id\'eaux homog\`enes int\'egralement
clos pour la compl\'exit\'e $1$. Soit $S$ le semi-groupe des poids de $A$.
Dans ce th\'eor\`eme, nous consid\'erons
des couples $(P,\widetilde{\mathfrak{D}})$ v\'erifiant les conditions suivantes.
\begin{enumerate}
\item[\rm (i)] $P$ est un poly\`edre de $\rm Pol_{\sigma^{\vee}}(\it S)$;
\item[\rm (ii)] $\widetilde{\mathfrak{D}} = \sum_{z\in C}\widetilde{\Delta}_{z}\cdot z$ 
est un diviseur $\widetilde{\omega}^{\vee}$-poly\'edral
o\`u  $\widetilde{\omega}\subset M_{\mathbb{Q}}\times\mathbb{Q}$ v\'erifie 
\begin{eqnarray*}
\widetilde{\omega}\cap M_{e} = (eP)\cap M, 
\end{eqnarray*}
pour tout entier $e\geq 0$;
\item[\rm (iii)] Pour tout 
$z\in C$, $\Delta_{z}$ est la projection orthogonale de $\widetilde{\Delta}_{z}$ parall\`element \`a 
$\mathbb{Q}\,(0_{M},1)$
et tout sommet $(v,p)\in\widetilde{\Delta}_{z}$ satisfait $p\leq 0$; 
\item[\rm (iv)]
Pour tout $z\in C$, $\widetilde{\Delta}_{z}$ est 
l'intersection dans $\Delta_{z}\times\mathbb{Q}$
d'un nombre fini de demi-espaces 
\begin{eqnarray*}
\Delta_{m,e} :=\left\{\, (v,p)\in N_{\mathbb{Q}}\times\mathbb{Q},\,\,m(v)+p\geq e\,\right\}  
\end{eqnarray*}
avec $e\in\mathbb{Z}$ et $m\in P\cap M$ tel que les sections globales du faisceau 
$\mathcal{O}_{C}(\lfloor\widetilde{\mathfrak{D}}(m,1)\rfloor)$ engendrent le germe
$\mathcal{O}_{C}(\lfloor\widetilde{\mathfrak{D}}(m,1)\rfloor)_{z}$.
\end{enumerate}

\begin{theorem}
Soit $\mathfrak{D} = \sum_{z\in C}\Delta_{z}\cdot z$ un diviseur $\sigma$-poly\'edral
propre et soit $A = A[C,\mathfrak{D}]$.
Alors il existe une bijection entre l'ensemble des id\'eaux homog\`enes int\'egralement clos de $A$
et l'ensemble des couples $(P,\widetilde{\mathfrak{D}})$ v\'erifiant $\rm (i) - \rm (iv)$.
Cette correspondance est donn\'ee par
\begin{eqnarray}
(P,\widetilde{\mathfrak{D}})\mapsto I = \bigoplus_{m\in P\cap M}
H^{ 0}(C,\mathcal{O}_{C}(\lfloor \widetilde{\mathfrak{D}}(m, 1 )\rfloor ))\chi^{m}. 
\end{eqnarray}
De plus, sous cette correspondance, l'alg\`ebre $A[C,\widetilde{\mathfrak{D}}]$ s'identifie 
\`a la normalisation de l'alg\`ebre de Rees $A[It]$.
\end{theorem}
\begin{proof}
Soit $(P,\widetilde{\mathfrak{D}})$ un couple v\'erifiant les conditions $\rm (i)-\rm (iv)$.
Montrons que $\widetilde{\mathfrak{D}}$ est propre. Soient
$z_{1},\ldots, z_{s}$ les points du support de $\widetilde{\mathfrak{D}}$.
D'apr\`es $\rm (iv)$, on a pour $j = 1,\ldots,s$,
\begin{eqnarray*}
\widetilde{\Delta}_{z_{j}} = (\Delta_{z}\times\mathbb{Q})\cap\bigcap_{i = 1}^{r_{j}}\Delta_{m_{ij},e_{ij}} 
\end{eqnarray*}
avec $e_{ij}\in\mathbb{Z}$ et $m_{ij}\in P\cap M$ tel qu'il existe $f_{ij}\chi^{(m_{ij},1)}\in 
A[C,\widetilde{\mathfrak{D}}]$ homog\`ene satisfaisant
 $\nu_{z_{j}}(f_{ij}) = h_{\widetilde{\Delta}_{z_{j}}}(m_{ij},1)$.
Consid\'erons l'alg\`ebre $B$ engendr\'ee par $A$ et par les \'el\'ements
\begin{eqnarray}
 f_{ij}\chi^{(m_{ij},1)},\,\,1\leq j\leq s,\, 1\leq i\leq r_{j}. 
\end{eqnarray}
Par le th\'eor\`eme $4.2$, l'anneau $A[C,\widetilde{\mathfrak{D}}]$ est la normalisation de $B$.
Ce qui donne la propret\'e de $\widetilde{\mathfrak{D}}$.

Les lemmes $4.4$ et $4.5$
montrent que l'application $(8)$ est bien d\'efinie.
Ainsi, $A[C,\widetilde{\mathfrak{D}}] = \overline{A[It]}$ 
o\`u $t = \chi^{(0,1)}$. Par l'\'egalit\'e $(6)$ de $3.1$, 
on d\'eduit l'injectivit\'e.

Montrons la surjectivit\'e. Soit $I$ un id\'eal int\'egralement clos de $A$
engendr\'e par des \'el\'ements homog\`enes $f_{1}\chi^{m_{1}},\ldots, f_{r}\chi^{m_{r}}$.
Consid\'erons le couple $(P,\widetilde{\mathfrak{D}})$ obtenu \`a partir du th\'eor\`eme $4.2$. 
Des lemmes $4.1$, $4.4$ et $4.5$, 
on obtient $\rm (i)$, $\rm (ii)$, $\rm (iii)$. Montrons $\rm (iv)$. Pour un point $z\in C$,
\'ecrivons
\begin{eqnarray*}
\widetilde{\Delta}_{z} = (\Delta_{z}\times\mathbb{Q})\cap\bigcap_{i = 1}^{r}\Delta_{m_{i},e_{i}} 
\end{eqnarray*}
avec pour tout $i$, $e_{i}:=-\nu_{z}(f_{i})$. Soit $E$ un sous-ensemble minimal de $\{1,\ldots, r\}$
tel que 
\begin{eqnarray*}
\widetilde{\Delta}_{z} = (\Delta_{z}\times\mathbb{Q})\cap \bigcap_{i\in E}\Delta_{m_{i},e_{i}}. 
\end{eqnarray*}
Fixons un \'el\'ement $i'\in E$. Alors on a
\begin{eqnarray}
\{(v,p)\in N_{\mathbb{Q}}\times\mathbb{Q}, m_{i'}(v)+ p = e_{i'}\}\cap\widetilde{\Delta}_{z}\neq\emptyset. 
\end{eqnarray}
Donc $h_{\widetilde{\Delta}_{z}}(m_{i'},1) = e_{i'} = -\nu_{z}(f_{i'})$. D'o\`u 
la condition $\rm (iv)$.
\end{proof}
\paragraph{}
Lorsque $C$ est affine, tout fibr\'e en droite au-dessus de $C$ est globalement engendr\'e.
Ainsi, comme cons\'equence imm\'ediate, on a le r\'esultat suivant. 

\begin{corollaire}
Supposons que $C$ est affine.
Alors
il existe une bijection entre l'ensemble des id\'eaux homog\`enes int\'egralement clos de $A$
et l'ensemble des couples $(P,\widetilde{\mathfrak{D}})$ v\'erifiant $\rm (ii), (iii)$ du th\'eor\`eme
$4.6$ et les conditions  suivantes.
\begin{enumerate}
\item[\rm (i)'] $P$ est un poly\`edre de $\rm Pol_{\sigma^{\vee}}(\it \sigma^{\vee}_{M})$;
\item[\rm (iv)']
Pour tout $z\in C$, le poly\`edre $\widetilde{\Delta}_{z}$ est 
l'intersection dans $\Delta_{z}\times\mathbb{Q}$
d'un nombre fini de demi-espaces 
\begin{eqnarray*}
\Delta_{m,e} :=\left\{\, (v,p)\in N_{\mathbb{Q}}\times\mathbb{Q},\,\,m(v)+p\geq e\,\right\}  
\end{eqnarray*}
avec $e\in\mathbb{Z}$ et $m\in P\cap M$.
\end{enumerate}
La correspondance est donn\'ee par l'application $(8)$ de $4.6$.
\end{corollaire}
\begin{remarque}
Soient $g_{2},g_{3}\in\mathbf{k}$ tels que 
$g_{2}^{3}-27g_{3}^{2}\neq 0$.
Consid\'erons la courbe elliptique $C\subset \mathbb{P}^{2}$
d'\'equation de Weierstrass $y^{2} = 4x^{3}-g_{2}x-g_{3}$,
$O$ son point \`a l'infini et $A$
l'alg\`ebre gradu\'ee
\begin{eqnarray*}
\bigoplus_{m\in \mathbb{N}}H^{0}\left(C,\mathcal{O}_{C}\left(m\cdot O\right)\right)\chi^{m}. 
\end{eqnarray*}
Soit $\widetilde{\omega}\subset\mathbb{Q}^{2}$ le c\^one  
\begin{eqnarray*}
\mathbb{Q}_{\geq 0}(1,0)+\mathbb{Q}_{\geq 0}(1,1), 
\end{eqnarray*}
$P = \mathbb{Q}_{\geq 1}$
et $\widetilde{\mathfrak{D}} := \widetilde{\Delta}\cdot O$
le diviseur poly\'edral sur $C$ avec
$\widetilde{\Delta} = (1,0) + \widetilde{\omega}^{\vee}$. Alors
$(P,\widetilde{\mathfrak{D}})$ v\'erifie les conditions $\rm (i)$,
$\rm (ii)$, $\rm (iii)$ et $\rm (iv)'$ de $4.6$ et $4.7$. Cependant
la condition $\rm (iv)$ n'est pas satisfaite. L'alg\`ebre $A[C,\widetilde{\mathfrak{D}}]$
est donc distincte de la normalisation de $A[It]$ o\`u $I$ est l'id\'eal provenant
de $(8)$ et $t:=\chi^{(0,1)}$. 
\end{remarque}

\section{Exemples d'id\'eaux homog\`enes normaux}
Dans cette section, nous consid\'erons l'alg\`ebre $A = A[C,\mathfrak{D}]$
comme dans la section $4$. Supposons que $C$ est affine et fixons 
un id\'eal homog\`ene int\'egralement clos $I\subset A$. 
Nous allons donner des conditions sur le couple 
$(P,\widetilde{\mathfrak{D}})$
associ\'e (voir $4.6$) pour que $I$ soit normal. Comme
d'habitude, nous notons $\widetilde{\Delta}_{z}$ le coefficient 
de $\widetilde{\mathfrak{D}}$ au point $z\in C$. 
\begin{notation} 
Pour tout $z\in C$, nous consid\'erons le sous-ensemble
\begin{eqnarray*}
\widetilde{P}_{z}: = \rm Conv\left(\it\left\{ (m,i)\in (P\cap M)\times \mathbb{Z}, \,\,
 h_{\widetilde{\rm \Delta\it}_{z}}
(m,\rm 1\it) \geq -i\right\}\right).
\end{eqnarray*}
On v\'erifie ais\'ement que $\widetilde{P}_{z}$ est un poly\`edre entier de $M_{\mathbb{Q}}\times\mathbb{Q}$.
Par ailleurs, pour un ouvert non vide $U\subset C$, 
$\widetilde{\mathfrak{D}}|U := \sum_{z\in U}\widetilde{\Delta}_{z}\cdot z$ est 
le diviseur poly\'edral sur $U$ obtenu par restriction de $\widetilde{\mathfrak{D}}$.
\end{notation}
L'assertion suivante est inspir\'ee de la description de $X[C,\widetilde{\mathfrak{D}}]$
comme vari\'et\'e toro\"idale (voir [LS,$\S 2.6$], [KKMS, Chapter $2$, $4$]). 

\begin{lemme}
Soit $z\in C$. Supposons que $C$ est affine et que $\mathfrak{D}$, 
$\widetilde{\mathfrak{D}}$ ont au plus le point $z$ dans leurs supports. 
Si le poly\`edre $\widetilde{P}_{z}$ est normal alors l'id\'eal
$I$ est normal. 
\end{lemme}
\begin{proof}
Fixons $e\in \mathbb{Z}_{>0}$. D\'eterminons un sous-ensemble g\'en\'erateur de 
$\overline{I^{e}}$. Pour tout vecteur $(m,i)\in M\times\mathbb{Z}$,
posons 
\begin{eqnarray*}
B_{(m,i)} := H^{0}(C,\mathcal{O}_{C}(-i\cdot z))\chi^{m}. 
\end{eqnarray*}
Si $m\in(eP)\cap M$ alors on a l'\'egalit\'e
\begin{eqnarray*}
H^{0}(C,\mathcal{O}_{C}(\lfloor \widetilde{\mathfrak{D}}(m,e)\rfloor))\chi^{m} = 
\bigcup_{i\in\mathbb{Z},\, i\geq - h_{z,e}(m)}
B_{(m,i)} 
\end{eqnarray*}
 o\`u $h_{z,e}(m) : = h_{\widetilde{\Delta}_{z}}(m,e)$. Donc
l'id\'eal $\overline{I^{e}}$ est engendr\'e par
\begin{eqnarray*}
\bigcup_{(m,i)\in (e\widetilde{P}_{z})\cap(M\times\mathbb{Z})} B_{(m,i)}.
\end{eqnarray*}
Montrons que pour tout 
$(m,i)\in (e\widetilde{P}_{z})\cap(M\times\mathbb{Z})$, la partie
$B_{(m,i)}$ est incluse dans $I^{e}$.
Fixons un tel couple $(m,i)$. Par normalit\'e de $\widetilde{P}_{z}$,
il existe
\begin{eqnarray*}
(m_{1},i_{1}),\ldots,(m_{e},i_{e})\in \widetilde{P}_{z}\cap (M\times\mathbb{Z})\,\,\,\,
\rm tels\,\,que\,\,\,\,\it\sum_{j = \rm 1 \it}^{e}(m_{j},i_{j}) = (m,i). 
\end{eqnarray*}
Pour chaque $j = 1,\ldots, e$, $B_{(m_{j},i_{j})}$ est contenu
dans $I$. Puisque la multiplication 
\begin{eqnarray*}
B_{(m_{1},i_{1})}\otimes\ldots \otimes B_{(m_{e},i_{e})}\rightarrow B_{(m,i)}
\end{eqnarray*}
est surjective, on a $B_{(m,i)}\subset I^{e}$. D'o\`u le r\'esultat.
\end{proof}
Le th\'eor\`eme suivant se d\'eduit du lemme pr\'ec\'edent par localisation.
Il peut \^etre vu comme un analogue de l'assertion $3.6$.
\begin{theorem}
Supposons que $C$ est affine. Soit $I\subset A = A[C,\mathfrak{D}]$
un id\'eal homog\`ene int\'egralement clos et $(P,\widetilde{\mathfrak{D}})$
le couple correspondant.
Si pour tout point $z$ appartenant au support
de $\widetilde{\mathfrak{D}}$, 
le poly\`edre $\widetilde{P}_{z}$ est
normal alors l'id\'eal $I$ est normal. Pour tout 
entier $e\geq n:=\dim\,N_{\mathbb{Q}}$, l'id\'eal $\overline{I^{e}}$ est normal. 
En particulier, tout id\'eal stable int\'egralement clos d'une 
$\mathbf{k}^{\star}$-surface affine non elliptique est normal.
\end{theorem}
\begin{proof}
Notons $z_{1},\ldots, z_{r}$ les points distincts du support
de $\widetilde{\mathfrak{D}}$. Par le th\'eor\`eme des restes
chinois, l'application 
\begin{eqnarray*}
\pi :\mathbf{k}[C]\rightarrow \mathbf{k}^{r},\,\,\,f\mapsto (f(z_{1}),\ldots, f(z_{r})) 
\end{eqnarray*}
est surjective. Soit $(e_{1},\ldots,e_{r})$ la base canonique de 
$\mathbf{k}^{r}$ et pour $i = 1,\ldots, r$, prenons
$f_{i}$ un \'el\'ement de $\pi^{-1}(\{e_{i}\})$.
Pour une fonction r\'eguli\`ere $f\in\mathbf{k}[C]$, nous
notons $C_{f} = C-V(f)$ son lieu de non annulation. 
Soient $f_{r+1},f_{r+2},\ldots,f_{s}\in\mathbf{k}[C]$ tels que 
\begin{eqnarray*}
U:= C-\{z_{1},\ldots, z_{r}\} = \bigcup_{i = r+1}^{s}C_{f_{j}}.
\end{eqnarray*}
Alors on a l'\'egalit\'e 
\begin{eqnarray*}
C = \bigcup_{i = 1}^{s}C_{f_{i}}. 
\end{eqnarray*}
Consid\'erons $z_{r+1},z_{r+2},\ldots ,z_{s}\in U$ des points
tels que $f_{i}(z_{i})\neq 0$ pour $i = r+1,\ldots, s$.
Fixons un entier $i\in\mathbb{N}$ tel
que $1\leq i\leq s$. Puisque $\widetilde{P}_{z_{i}}$ est un
poly\`edre entier normal, par le lemme $5.2$, l'id\'eal
\begin{eqnarray*}
I_{f_{i}}: =  \bigoplus_{m\in P\cap M}
H^{0}(C_{f_{i}},\mathcal{O}_{C}(\lfloor\widetilde{\mathfrak{D}}(m,1)\rfloor ))\chi^{m} = 
I\otimes_{\mathbf{k}[C]}\mathbf{k}[C_{f_{i}}]\subset A\otimes_{\mathbf{k}[C]}\mathbf{k}[C_{f_{i}}]
\end{eqnarray*}
correspondant au couple $(P,\widetilde{\mathfrak{D}}|C_{f_{i}})$ est normal.
Donc nous avons 
\begin{eqnarray}
A[C_{f_{i}},\widetilde{\mathfrak{D}}|C_{f_{i}}] = 
\overline{(A\otimes_{\mathbf{k}[C]}\mathbf{k}[C_{f_{i}}])
[I_{f_{i}}t]} =
(A\otimes_{\mathbf{k}[C]}\mathbf{k}[C_{f_{i}}])
[I_{f_{i}}t] = A[It]\otimes_{\mathbf{k}[C]}
\mathbf{k}[C_{f_{i}}].
\end{eqnarray}
\'Ecrivons 
\begin{eqnarray*}
A[It] = \bigoplus_{(m,e)\in\widetilde{\sigma}^{\vee}_{M\times\mathbb{Z}}}
I_{(m,e)}\chi^{m}t^{e} 
\end{eqnarray*}
avec $I_{(m,e)}\subset\mathbf{k}(C)$, pour tout $(m,e)$.
Alors par $(11)$, le semi-groupe des poids
de $A[It]$ est $\widetilde{\sigma}^{\vee}_{M\times\mathbb{Z}}$.
Cela implique que chaque $I_{(m,e)}$ est un id\'eal fractionnaire
de $\mathbf{k}[C]$. Donc pour tout vecteur $(m,e)\in 
\widetilde{\sigma}^{\vee}_{M\times\mathbb{Z}}$, il existe un
diviseur de Weil entier $D_{(m,e)}$ sur $C$ tel que
\begin{eqnarray*}
 I_{(m,e)} = H^{0}(C,\mathcal{O}_{C}(D_{(m,e)})).
\end{eqnarray*} 
Comme 
\begin{eqnarray*}
A[It]\otimes_{\mathbf{k}[C]}\mathbf{k}[C_{f_{i}}] =  
\bigoplus_{(m,e)\in\widetilde{\sigma}^{\vee}_{M\times\mathbb{Z}}}
H^{0}(C_{f_{i}},\mathcal{O}_{C}(D_{(m,e)}))\chi^{m}t^{e},
\end{eqnarray*}
on a d'une part,
\begin{eqnarray*}
\bigcap_{i = 1}^{s}A[It]\otimes_{\mathbf{k}[C]}\mathbf{k}[C_{f_{i}}]
=A[It] 
\end{eqnarray*}
et d'autre part,
\begin{eqnarray*}
\bigcap_{i = 1}^{s}A[C_{f_{i}},\widetilde{\mathfrak{D}}|C_{f_{i}}]
= A[C,\widetilde{\mathfrak{D}}], 
\end{eqnarray*}
on conclut par $(11)$ que $A[It] = A[C,\widetilde{\mathfrak{D}}]$. 
Cela donne la normalit\'e de $I$. Le reste de la preuve est une 
cons\'equence directe du th\'eor\`eme $3.6$. 
\end{proof}
La prochaine assertion est une traduction combinatoire de [RRV, Proposition $3.1$]
via la correspondance du th\'eor\`eme $3.5$.
\begin{lemme}
Soit $n\in\mathbb{N}$ un entier, notons $\sigma^{\vee} = \mathbb{Q}^{n+1}_{\geq 0}$ et
soit $P$ un $\sigma^{\vee}$-poly\`edre entier contenu dans $\sigma^{\vee}$.
Alors les assertions suivantes sont \'equivalentes.
\begin{enumerate}
\item[\rm (i)] Le poly\`edre est normal;
\item[\rm (ii)] Pour tout $s\in\{1,\ldots,n\}$, 
on a l'\'egalit\'e $(sP)\cap\mathbb{Z}^{n+1} = E_{[s,P]}$ (voir $3.2$).
\end{enumerate}
\end{lemme}
\paragraph{}
Comme application du th\'eor\`eme $5.3$, nous obtenons une
caract\'erisation de la normalit\'e pour une classe d'id\'eaux 
de l'alg\`ebre des polyn\^omes.
\begin{corollaire}
Soit $n\in\mathbb{Z}_{>0}$.
Consid\'erons l'alg\`ebre des polyn\^omes $\mathbf{k}^{[n+1]} = \mathbf{k}[x_{0},x_{1},\ldots,x_{n}]$ \`a $n+1$
variables munie de la $\mathbb{Z}^{n}$-graduation 
\begin{eqnarray*}
\mathbf{k}^{[n+1]} = \bigoplus_{(m_{1},\ldots,m_{r})\in\mathbb{N}^{n}}\mathbf{k}[x_{0}]x_{1}^{m_{1}}\ldots\, x_{n}^{m_{n}}
\end{eqnarray*}
et soit $I$ un id\'eal homog\`ene de $A$. Alors les assertions suivantes sont \'equivalentes.
\begin{enumerate}
\item[\rm (i)] L'id\'eal $I$ est normal;
\item[\rm (ii)] Pour tout $e\in\{1,\ldots, n\}$, l'id\'eal $I^{e}$ est int\'egralement clos.
\end{enumerate}
\end{corollaire}
\begin{proof}
Posons $C:=\mathbb{A}^{1} = \rm Spec\,\mathbf{k}[\it x_{\rm \,0}]$, $\sigma := \mathbb{Q}^{n}_{\geq 0}$ et
$M := \mathbb{Z}^{n}$.
Consid\'erons le diviseur $\sigma$-poly\'edral $\mathfrak{D}$ sur la courbe $C$ dont l'\'evaluation est
identiquement nulle. Alors on a  $\mathbf{k}^{[n+1]} = A = A[C,\mathfrak{D}]$. 

Montrons l'implication $\rm (ii)\Rightarrow (i)$.
Soit $(P,\widetilde{\mathfrak{D}})$ le couple correspondant \`a l'id\'eal $I$.
Notons $z_{1},\ldots, z_{r}$ les points distincts du support de $\widetilde{\mathfrak{D}}$.
Pour $i = 1,\ldots,r$, consid\'erons le polyn\^ome
\begin{eqnarray*}
f_{i}(x_{0}) = \prod_{j\neq i}(x_{0} - z_{j}).
\end{eqnarray*}
Alors pour tout $i$, nous avons 
\begin{eqnarray*}
I_{f_{i}}:=I\otimes_{\mathbf{k}[C]}\mathbf{k}[C_{f_{i}}] = 
\bigoplus_{(m,e)\in \widetilde{P}_{z_{i}}\cap (M\times\mathbb{Z})}
\mathbf{k}\left[\frac{1}{f_{i}}\right](x_{0}-z_{i})^{e}\chi^{m}
\end{eqnarray*} 
o\`u pour $m=(m_{1},\ldots,m_{r})$, $\chi^{m} := x_{1}^{m_{1}}\ldots x_{n}^{m_{n}}$.
Fixons $s\in\{1,\ldots, n\}$. Alors on a d'une part,
\begin{eqnarray*}
I_{f_{i}}^{s} = \bigoplus_{(m,e)\in E_{[s,\widetilde{P}_{z_{i}}]}}
\mathbf{k}\left[\frac{1}{f_{i}}\right](x_{0}-z_{i})^{e}\chi^{m}
\end{eqnarray*}
et d'autre part,
\begin{eqnarray*}
\overline{I_{f_{i}}^{s}} = \bigoplus_{m\in (sP)\cap M}H^{0}(C_{f_{i}},
\mathcal{O}(\lfloor \widetilde{\mathfrak{D}}(m,e)\rfloor))\chi^{m} =
\bigoplus_{(m,e)\in (s\widetilde{P}_{z_{i})}\cap (M\times\mathbb{Z})}
\mathbf{k}\left[\frac{1}{f_{i}}\right](x_{0}-z_{i})^{e}\chi^{m}, 
\end{eqnarray*}
comparer avec [HS, Proposition $1.1.4$].
Puisque $I^{s}$ est int\'egralement clos, l'id\'eal $I_{f_{i}}^{s}\subset A_{f_{i}}$ l'est encore.
Ce qui donne par les \'egalit\'es pr\'ec\'edentes,
\begin{eqnarray*}
(s\widetilde{P}_{z_{i}})\cap (M\times\mathbb{Z}) 
= E_{[s,\widetilde{P}_{z_{i}}]}.
\end{eqnarray*}
En appliquant le lemme $5.4$, on d\'eduit que $\widetilde{P}_{z_{i}}$ est normal. 
Par le th\'eor\`eme $5.3$, on obtient que $I$ est normal. D'o\`u l'implication 
$\rm (ii)\Rightarrow (i)$. La r\'eciproque est ais\'ee.
\end{proof}

\paragraph{}
\paragraph{}

\texttt{Adresse :
\\ Universit\'e Grenoble I, Institut Fourier, 
\\ UMR 5582  CNRS-UJF, BP 74, 
\\ 38402 St. Martin d'H\`eres c\'edex, France.}

\paragraph{}

\texttt{Courriel :
\\ Kevin.Langlois@ujf-grenoble.fr.}

\end{document}